\documentclass[11pt]{article}  

\usepackage{amssymb,amsmath,amsthm}
\usepackage[alphabetic]{amsrefs}
\usepackage[
  colorlinks=true,
  linkcolor=blue,
  citecolor=blue,
  urlcolor=blue
]{hyperref}

\usepackage{booktabs}
\usepackage{multirow}

\usepackage[a4paper, left=20mm, right=20mm, top=30mm, bottom=25mm]{geometry}

\newtheorem{theo}{Theorem}[section]
\newtheorem{theoremalpha}{Theorem} 

\newtheorem{lemm}[theo]{Lemma}
\newtheorem{coro}[theo]{Corollary}
\newtheorem{prop}[theo]{Proposition}

\numberwithin{equation}{section}

\theoremstyle{definition}
\newtheorem{defi}[theo]{Definition}
\newtheorem{rema}[theo]{Remark}

\newtheorem*{MainTheoremA}{Theorem \ref{theo:inequality_for_Killing_number}}
\newtheorem*{MainTheoremB}{Theorem \ref{theo:RKS_on_Sasakian_eta_Einstein}}

\newcommand{\C}{\mathbb{C}}
\newcommand{\R}{\mathbb{R}}

\DeclareMathOperator{\ad}{\mathrm{ad}}
\DeclareMathOperator{\Ric}{\mathrm{Ric}}
\DeclareMathOperator{\scal}{\mathrm{Scal}}
\DeclareMathOperator{\Id}{\mathrm{Id}}
\DeclareMathOperator{\tr}{\mathrm{tr}}
\DeclareMathOperator{\bw}{{\textstyle\bigwedge}}

\newcommand{\Spin}{\mathrm{Spin}}
\newcommand{\SL}{\mathrm{SL}}
\newcommand{\SO}{\mathrm{SO}}
\newcommand{\SU}{\mathrm{SU}}
\newcommand{\ISO}{\mathrm{ISO}}

\newcommand{\g}{\mathfrak{g}}
\renewcommand{\sl}{\mathfrak{sl}}
\newcommand{\su}{\mathfrak{su}}
\newcommand{\so}{\mathfrak{so}}
\newcommand{\iso}{\mathfrak{iso}}
\newcommand{\heis}{\mathfrak{heis}}

\title{\textbf{Generalized Killing spinors \\ associated with the Ricci tensor}}
\date{\today}
\author{Natsuki Imada}

\begin{document}
\maketitle

\begin{abstract}
  In this paper, we introduce the notion of Ricci Killing spinors on Riemannian spin manifolds, which form a class between generalized Killing spinors and standard Killing spinors.
  We establish a rigidity result for Ricci Killing spinors and characterize simply connected Sasakian $\eta$-Einstein spin manifolds admitting Ricci Killing spinors.
  We also give a complete classification of Ricci Killing spinors on simply connected $3$-dimensional unimodular Lie groups with left-invariant Riemannian metrics.
  In particular, we obtain new examples of generalized Killing spinors which are not Killing spinors.
  Finally, for the non-Killing Ricci Killing spinors obtained in the Lie group case, we explicitly determine the corresponding Ricci-flat ambient manifolds carrying parallel spinors.
\end{abstract}

\noindent \textbf{Keywords: }{Generalized Killing spinors; Sasakian manifolds; Lie groups; parallel spinors.}\\
\noindent \textbf{2020 Mathematics Subject Classification:} {Primary 53C27; Secondary 53C25, 53C30.}


\section{Introduction}

Killing spinors are special sections of the spinor bundle on a spin manifold that satisfy a certain differential equation.
They have been extensively studied in differential geometry and mathematical physics due to their connections with special holonomy, supersymmetry, and Einstein metrics.
In this paper, we introduce a new class of spinors called \textit{Ricci Killing spinors}, which generalize the notion of Killing spinors by incorporating the Ricci curvature of the manifold.

A generalized Killing spinor or Cauchy spinor on a spin manifold $(M,g)$ is a non-zero spinor field $\psi \in \Gamma(\Sigma M)$ that satisfies the equation
\[\nabla_X \psi = A(X) \cdot \psi \quad \text{for all } X \in TM,\]
where $A$ is a symmetric endomorphism of the tangent bundle $TM$ and $\cdot$ denotes the Clifford multiplication.
When $A$ is a constant multiple of the identity map, $\psi$ is called a Killing spinor \cite{BFGK}.

Generalized Killing spinors are strongly related to hypersurfaces in Riemannian manifolds admitting parallel spinors.
If a Riemannian spin manifold $\tilde{M}$ admits a parallel spinor $\tilde{\psi}$, then for any oriented hypersurface $M \subset \tilde{M}$, $\tilde{\psi}|_M$ is a generalized Killing spinor with respect to $A = \frac{1}{2}W$, where $W$ is the Weingarten tensor of $M$ \cite{BHMMM}.
The converse is partially true. Suppose that a Riemannian spin manifold $(M,g)$ admits a generalized Killing spinor $\psi$ with a symmetric endomorphism $A$.
If either $A$ is a Codazzi tensor \cite{BGM05}, or both $g$ and $A$ are real-analytic \cite{AMM13}, then $M$ can be isometrically immersed as a hypersurface in a Riemannian spin manifold $\tilde{M}$ admitting a parallel spinor $\tilde{\psi}$ such that the Weingarten tensor is $W = 2A$ and $\psi = \tilde{\psi}|_M$.
In particular, if $\psi$ is a real Killing spinor, $\tilde{M}$ can be chosen as the Riemannian cone over $M$.
This is exactly B\"ar's cone construction \cite{Bar93}.

In the search for examples of generalized Killing spinors that are not Killing spinors, several important results have been obtained.
For $3$- and $7$-dimensional spheres, explicit examples of such generalized Killing spinors were given in \cite{MS14b} (see also \cite{MS14c}).
Furthermore, every generalized Killing spinor on a standard sphere of dimension $2$, $4$, $5$, or $8$ is a Killing spinor \cites{MS14a,MS14b}.
The properties of generalized Killing spinors on the $3$-dimensional sphere were further investigated in detail in \cites{MS14a,FM22}.
However, despite these efforts, a complete classification of generalized Killing spinors remains an open problem even on the $3$-dimensional sphere.
The existence and construction of generalized Killing spinors have also been studied on other specific geometric structures.
For instance, they were investigated on $7$-dimensional $3$-Sasakian manifolds \cite{AF10} and, more recently, on $3$-dimensional Lie groups \cite{Art25}.

In the present paper, we introduce the notion of Ricci Killing spinors as a special class of generalized Killing spinors where the endomorphism $A$ is related to the Ricci curvature of the manifold.
Specifically, we define a Ricci Killing spinor as a non-zero spinor field $\psi$ that satisfies
\[\nabla_X \psi = \mu\Ric(X) \cdot \psi \quad \text{for all } X \in TM,\]
for some non-zero constant $\mu \in \R\setminus\{0\}$.
Here, $\Ric$ is the Ricci tensor viewed as an endomorphism of the tangent bundle.
The class of Ricci Killing spinors includes that of Killing spinors.
Then a natural question arises: \textit{under what conditions does a Ricci Killing spinor reduce to a Killing spinor?}
In this paper, we first answer this question by establishing a geometric inequality relating the constant $\mu$ and the scalar curvature. Specifically, we prove the following theorem:

\begin{MainTheoremA}
Let $(M, g)$ be an $n$-dimensional Riemannian spin manifold ($n \geq 2$) admitting a Ricci Killing spinor with a constant $\mu$. 
Then the scalar curvature $\scal$ is constant.
If $\scal = 0$ then $(M, g)$ is Ricci-flat and the Ricci Killing spinor is parallel. If $\scal \neq 0$ then the following inequality holds:
\[\mu^2 \geq \frac{n}{4(n-1)\scal}.\]
Equality holds if and only if $(M, g)$ is Einstein, namely, the Ricci Killing spinor is a Killing spinor.
\end{MainTheoremA}

Next, we investigate the local structure of manifolds admitting Ricci Killing spinors.
It is a classical result that Riemannian spin manifolds admitting non-parallel Killing spinors are locally irreducible.
However, as we will see in Section \ref{sec:GKS}, generalized Killing spinors can exist on locally reducible manifolds.
We prove that the existence of Ricci Killing spinors on product manifolds imposes a strong restriction on the geometry of the factors.
More precisely, if a Riemannian spin manifold contains an open subset isometric to a Riemannian product of two non-Ricci-flat Riemannian manifolds, then it does not admit a Ricci Killing spinor.

By taking a product with a Riemannian spin manifold admitting a parallel spinor, one can easily construct Ricci Killing spinors that are not Killing spinors.
Therefore, it is natural to ask whether one can find examples of Ricci Killing spinors on non-product manifolds.
To address this question, we first consider Ricci Killing spinors on Sasakian $\eta$-Einstein manifolds, and we obtain the following characterization.

\begin{MainTheoremB}
  Let $(M^{2m+1}, g, \xi, \eta, \varphi)\,\,(m \geq 1)$ be a simply connected Sasakian $\eta$-Einstein spin manifold.
  Then $M$ admits a Ricci Killing spinor if and only if $M$ is either transversely Ricci-flat or Sasakian-Einstein.
  Furthermore, Ricci Killing spinors on transversely Ricci-flat Sasakian spin manifolds are transversely parallel.
\end{MainTheoremB}

In particular, simply connected transversely Ricci-flat Sasakian spin manifolds provide examples of Ricci Killing spinors which are not Killing spinors.
We determine the possible values of $\mu$ and the dimension of the space of Ricci Killing spinors for each $\mu$.

We next investigate Ricci Killing spinors on $3$-dimensional Lie groups. Artacho \cite{Art25} classified invariant generalized Killing spinors on $3$-dimensional Lie groups.
We first determine which of these invariant generalized Killing spinors are Ricci Killing spinors.
We then consider spinors which are not necessarily invariant and obtain a complete classification of Ricci Killing spinors on simply connected $3$-dimensional unimodular Lie groups endowed with left-invariant Riemannian metrics.
The classification is summarized in Table~\ref{table:RKS_on_3dim_unimodular_Lie_groups}.

\begin{table}[htbp]
  \centering
  \setlength{\tabcolsep}{1.1em}
  \caption{Ricci Killing spinors on simply connected 3-dimensional unimodular Lie groups}
  \begin{tabular}{ccc}
    \toprule
    Lie algebra $\g$ & Milnor constants & Ricci Killing spinors \\
    \midrule
    \midrule

    \multirow{3}{*}{$\su(2)$}
    & \multirow{2}{*}{$\lambda_1 = \lambda_2 = \lambda_3 =: \lambda \neq 0$}
    & invariant, Killing, $\mu = \frac{1}{2\lambda}$ \\
    &
    & non-invariant, Killing, $\mu = -\frac{1}{2\lambda}$ \\
    & otherwise
    & none \\

    \midrule

    $\sl(2,\R)$
    & arbitrary
    & none \\

    \midrule

    \multirow{2}{*}{$\iso(2)$}
    & $\lambda_1 = \lambda_2 \neq 0,\ \lambda_3 = 0$
    & non-invariant, parallel, any $\mu \neq 0$ \\
    & $\lambda_1\lambda_2 > 0,\ \lambda_1 \neq \lambda_2,\ \lambda_3 = 0$
    & non-invariant, non-Killing, $\mu = -\frac{1}{2(\lambda_1 + \lambda_2)}$ \\

    \midrule

    \multirow{2}{*}{$\iso(1,1)$}
    & $\lambda_1\lambda_2 < 0,\ \lambda_1 + \lambda_2 \neq 0,\ \lambda_3 = 0$
    & non-invariant, non-Killing, $\mu = -\frac{1}{2(\lambda_1 + \lambda_2)}$ \\
    & $\lambda_1\lambda_2 < 0,\ \lambda_1 + \lambda_2 = 0,\ \lambda_3 = 0$
    & none \\

    \midrule

    $\heis_3$
    & $\lambda_1 = \lambda \neq 0,\ \lambda_2 = \lambda_3 = 0$
    & invariant, non-Killing, $\mu = -\frac{1}{2\lambda}$ \\

    \midrule

    $\R^3$
    & $\lambda_1 = \lambda_2 = \lambda_3 = 0$
    & invariant, parallel, any $\mu \neq 0$ \\

    \bottomrule
  \end{tabular}
  \label{table:RKS_on_3dim_unimodular_Lie_groups}
\end{table}

In particular, besides the Killing and parallel cases, non-Killing Ricci Killing spinors occur on the Heisenberg group and on certain left-invariant metrics on the simply connected Lie groups with Lie algebras $\iso(2)$ and $\iso(1,1)$.
The Ricci Killing spinors in the Heisenberg case are invariant, whereas those in the non-flat $\iso(2)$ and $\iso(1,1)$ cases are non-invariant.
We also give explicit expressions for the latter spinors.

Finally, we return to the relation between generalized Killing spinors and hypersurfaces in manifolds admitting parallel spinors.
For the left-invariant metrics considered above, both the metric and the associated endomorphism $A = \mu\Ric$ are real analytic, so the ambient construction of \cite{AMM13} applies.
We explicitly determine the ambient Riemannian manifold $\tilde{M}$ for the non-Killing Ricci Killing spinors obtained in the $\heis_3$, $\iso(2)$, and $\iso(1,1)$ cases.
The $4$-dimensional manifolds $\tilde{M}$ are Ricci-flat and carry parallel spinors whose restrictions to the hypersurfaces $M \subset \tilde{M}$ are the given Ricci Killing spinors. Moreover, the Weingarten tensor $W$ of $M$ in $\tilde{M}$ satisfies $W = 2\mu\Ric$.
Thus, these constructions provide explicit Ricci-flat spin manifolds carrying parallel spinors and containing $3$-dimensional non-Einstein hypersurfaces whose Weingarten tensors are proportional to their Ricci tensors.

This paper is organized as follows. In Section \ref{sec:GKS}, we briefly review the basic properties of generalized Killing spinors and discuss their behavior on product manifolds.
In Section \ref{sec:RKS}, we derive the fundamental equations for Ricci Killing spinors, prove Theorem \ref{theo:inequality_for_Killing_number}, and investigate the local structure of manifolds admitting Ricci Killing spinors.
In Section \ref{sec:RKS_on_Sasakian}, we focus on Sasakian $\eta$-Einstein manifolds, prove Theorem \ref{theo:RKS_on_Sasakian_eta_Einstein}, and investigate the relation between Ricci Killing spinors and transversely parallel spinors. 
In Section \ref{sec:RKS_on_3dim_Lie_groups}, we study Ricci Killing spinors on $3$-dimensional Lie groups, obtain a complete classification in the simply connected unimodular case, and give explicit expressions for the non-invariant non-Killing Ricci Killing spinors.
Finally, in Section \ref{sec:Ambient_metrics_of_RKS}, we determine the ambient metrics corresponding to the non-Killing Ricci Killing spinors obtained in the $3$-dimensional Lie group case.

\section{Generalized Killing spinors}\label{sec:GKS}

Let $(M,g)$ be an $n$-dimensional Riemannian spin manifold, and let
$\Sigma M$ denote its complex spinor bundle.
If $n$ is odd, there are two inequivalent choices of Clifford multiplication
on $\Sigma M$.
Throughout this paper, we choose the one for which the complex volume element
\[
\omega^{\C}
=
i^{\lfloor (n+1)/2\rfloor}e_1\cdot\ldots\cdot e_n
\]
acts as $\Id$ on $\Sigma M$.
If $n$ is even, the complex spinor bundle carries, up to equivalence, a unique
Clifford multiplication.

We first review the basic properties of generalized Killing spinors, following \cite{MS14a}.
For the basics of spin geometry, we refer the reader to \cites{LM89,BHMMM}.
\begin{defi}
A non-zero spinor field $\psi \in \Gamma(\Sigma M)$ on a Riemannian spin manifold $(M,g)$ is called a \textit{generalized Killing spinor} if there exists a symmetric endomorphism $A$ of the tangent bundle $TM$ such that
\[\nabla_X \psi = A(X) \cdot \psi \quad \text{for all } X \in TM.\]
Here, an endomorphism $A$ is said to be symmetric if $g(A(X), Y) = g(X, A(Y))$ for all $X, Y \in TM$.
\end{defi}
If $A$ is a constant multiple of the identity map, then $\psi$ is called a Killing spinor, and in this case $(M, g)$ is Einstein. In particular, if $A = 0$, then $\psi$ is called a parallel spinor, and in this case $(M, g)$ is Ricci-flat.
 However, in general, a Riemannian spin manifold admitting a generalized Killing spinor is not necessarily Einstein. We give examples of such manifolds in Section \ref{sec:RKS_on_Sasakian}.

By the same argument as in the case of Killing spinors, we can show that the following lemma holds for generalized Killing spinors as well.
\begin{lemm}\label{lemm:GKS_has_no_zeroes}
The norm of a generalized Killing spinor is constant. In particular, a generalized Killing spinor has no zeroes.
\end{lemm}

Let $R^{\Sigma M}(X, Y) = [\nabla_X, \nabla_Y] - \nabla_{[X, Y]}$ be the curvature tensor of the spin connection $\nabla$ on the spinor bundle $\Sigma M$. 
If $\psi$ is a generalized Killing spinor with a symmetric endomorphism $A$, then we have
  \[R^{\Sigma M}(X,Y) \psi = (d^{\nabla} A)(X,Y) \cdot \psi - 2A(X) \wedge A(Y) \cdot \psi.\]
Here, $d^{\nabla} A$ is the exterior covariant derivative of $A$ defined by $(d^{\nabla} A)(X,Y) = (\nabla_X A)(Y) - (\nabla_Y A)(X)$.
Let $R \colon \bw^2 TM \to \bw^2 TM$ denote the curvature operator defined by
\[g(R(X \wedge Y), Z \wedge W) = g(R(X, Y)Z, W).\]
Then the curvature of the spin connection satisfies
\[R^{\Sigma M}(X, Y)\varphi = \frac{1}{2}R(X \wedge Y) \cdot \varphi\]
for any $\varphi \in \Gamma(\Sigma M)$.
Thus, the formula for $R^{\Sigma M}(X, Y)\psi$ can be written as follows.
\begin{prop}\label{prop:curv_to_GKS}
  Let $\psi \in \Gamma(\Sigma M)$ be a generalized Killing spinor on a Riemannian spin manifold $(M,g)$ with a symmetric endomorphism $A$.
  Then the following equation holds:
  \[\frac{1}{2}R(X\wedge Y) \cdot \psi = (d^{\nabla} A)(X,Y) \cdot \psi - 2A(X) \wedge A(Y) \cdot \psi\]
  for all $X, Y \in TM$. 
\end{prop}
\begin{rema}\label{rema:curv_to_GKS_dim3}
When $M$ is three-dimensional, Proposition \ref{prop:curv_to_GKS} can be reformulated in a purely tensorial form.
Let $\ast$ denote the Hodge star operator.
Since $V \cdot \psi = (\ast V) \cdot \psi$ for every $V \in TM$ and the action of $\bw^2 T_pM$ on a non-zero spinor $\psi_p$ is injective, Proposition \ref{prop:curv_to_GKS} implies that any generalized Killing spinor associated with a symmetric endomorphism $A$ satisfies
\[\frac{1}{2}R(X \wedge Y) = \ast(d^{\nabla} A)(X, Y) - 2A(X) \wedge A(Y).\]
Conversely, if this equation holds, then the connection $\nabla^A_X := \nabla_X - A(X) \cdot$ on $\Sigma M$ is flat. Hence, generalized Killing spinors associated with $A$ exist locally (see also \cite{FM22}).
\end{rema}

By taking suitable contractions of the curvature identity in Proposition \ref{prop:curv_to_GKS}, one obtains further basic identities for generalized Killing spinors.

\begin{prop}[\cite{MS14a}]\label{prop:basic_properties_of_GKS}
Let $\psi \in \Gamma(\Sigma M)$ be a generalized Killing spinor on a Riemannian spin manifold $(M,g)$ with a symmetric endomorphism $A$. We set $a = \tr(A) \in C^{\infty}(M)$. Then the following equations hold:
\begin{enumerate}
  \item $\sum_{i=1}^n e_i \wedge (\nabla_{e_i} A)(X) \cdot \psi = \left(\frac{1}{2}\Ric(X) + 2A^2(X) - 2a A(X)\right) \cdot \psi \quad \forall X \in TM$,
  \item $\delta A + d a = 0$,
  \item $a^2 - |A|^2 = \frac{1}{4}\scal$,
\end{enumerate}
where $\{e_i\}_{i=1}^n$ is a local orthonormal frame of $TM$, $\delta A = -\sum_{i=1}^n (\nabla_{e_i} A)(e_i)$ is the divergence of $A$, $|A|^2 = \tr(A^2) = \sum_{i=1}^n g(A(e_i), A(e_i))$ is the squared norm of $A$, and $\scal$ is the scalar curvature of $(M, g)$.
\end{prop}
We note that Proposition \ref{prop:basic_properties_of_GKS} was proved for real spinor bundles in \cite{MS14a}, but the same identities hold for complex spinor bundles as well (see also \cite{AMM13}). 

Finally, we briefly remark on the irreducibility of manifolds admitting generalized Killing spinors. It is known that if a Riemannian spin manifold admits a non-parallel Killing spinor, then it is locally irreducible (i.e., it cannot be decomposed locally as a Riemannian product of two or more manifolds) \cite{BFGK}. However, this is not the case for generalized Killing spinors.
For example, we have the following proposition.
\begin{prop}\label{prop:GKS_on_product_manifolds}
  Let $M$ and $N$ be Riemannian spin manifolds, with $M$ of even dimension.
  If $M$ admits a generalized Killing spinor and $N$ admits a parallel spinor, then the product manifold $M \times N$ admits a generalized Killing spinor.
\end{prop}
\begin{proof}
  If $M$ is even-dimensional, then the spinor bundle of $M \times N$ is isomorphic to the tensor product $\Sigma M \otimes \Sigma N$ of the spinor bundles of $M$ and $N$ (see \cite{Bar98}).
  Hence, for a spinor field $\psi_M \in \Gamma(\Sigma M)$ and a spinor field $\psi_N \in \Gamma(\Sigma N)$, we can define a spinor field $\psi \in \Gamma(\Sigma (M \times N))$ by $\psi = \psi_M \otimes \psi_N$. We note that the Clifford multiplication on $M \times N$ is given by
  \begin{align*}
    X \cdot (\psi_M \otimes \psi_N) &= (X \cdot \psi_M) \otimes \psi_N,\\
    Y \cdot (\psi_M \otimes \psi_N) &= (\omega^{\C}_M \cdot \psi_M) \otimes (Y \cdot \psi_N),
  \end{align*}
  for $X \in TM$ and $Y \in TN$, where $\omega^{\C}_M$ denotes the complex volume element of $M$.
  The spin connection on $M \times N$ is given by
  \[\nabla_{X + Y}(\psi_M \otimes \psi_N) = (\nabla^M_X \psi_M) \otimes \psi_N + \psi_M \otimes (\nabla^N_Y \psi_N) \quad (X \in TM, Y \in TN).\]
  If $\psi_M$ is a generalized Killing spinor on $M$ with a symmetric endomorphism $A$ and $\psi_N$ is a parallel spinor on $N$, then we have
  \begin{align*}
    \nabla_{X + Y}(\psi_M \otimes \psi_N) &= (\nabla^M_X \psi_M) \otimes \psi_N + \psi_M \otimes (\nabla^N_Y \psi_N)\\
    &= (A(X) \cdot \psi_M) \otimes \psi_N\\
    &= A(X) \cdot (\psi_M \otimes \psi_N).
  \end{align*}
  Therefore, if we define a symmetric endomorphism $\tilde{A}$ of $TM \oplus TN$ by $\tilde{A}(X + Y) = A(X)$, then $\psi$ is a generalized Killing spinor on $M \times N$ with the symmetric endomorphism $\tilde{A}$.
\end{proof}

\section{Ricci Killing spinors}\label{sec:RKS}
In this section, we introduce the notion of Ricci Killing spinors and investigate their basic properties.
We also characterize when a Ricci Killing spinor reduces to a Killing spinor.
\begin{defi}
A non-zero spinor field $\psi \in \Gamma(\Sigma M)$ on a Riemannian spin manifold $(M,g)$ is called a \textit{Ricci Killing spinor} if there exists a non-zero constant $\mu \in \R\setminus\{0\}$ such that
\[\nabla_X \psi = \mu\Ric(X) \cdot \psi \quad \text{for all } X \in TM,\]
where $\Ric$ is the Ricci tensor viewed as an endomorphism of the tangent bundle.
\end{defi}

Real Killing spinors and parallel spinors are special cases of Ricci Killing spinors. Indeed, if $\psi$ is a Killing spinor satisfying $\nabla_X \psi = \lambda X \cdot \psi$ for some constant $\lambda \in \R\setminus\{0\}$, then the manifold is Einstein with $\Ric =  4(n-1)\lambda^2 \Id$. The Killing spinor equation implies that $\psi$ is also a Ricci Killing spinor with $\mu = (4(n-1)\lambda)^{-1}$.
On the other hand, if $\psi$ satisfies $\nabla_X \psi = 0$ for all $X \in TM$, then the manifold is Ricci-flat. In this case, $\psi$ is also a Ricci Killing spinor with any non-zero constant $\mu$.

\begin{rema}
We exclude $\mu = 0$ from the definition, since the parallel case is already included as described above.
\end{rema}

Applying Proposition \ref{prop:basic_properties_of_GKS} to $A = \mu \Ric$, we obtain the following identities.
\begin{prop}\label{prop:Basic_properties_of_RKS}
Let $\psi \in \Gamma(\Sigma M)$ be a Ricci Killing spinor on a Riemannian spin manifold $(M,g)$ with a constant $\mu$. Then the following equations hold:
\begin{enumerate}
\item $\mu \sum_{i=1}^n e_i \wedge (\nabla_{e_i} \Ric)(X) \cdot \psi = \left(\frac{1}{2}\Ric(X) + 2\mu^2\Ric^2(X) - 2\mu^2\scal\Ric(X)\right) \cdot \psi \quad \forall X \in TM$,
\item $\delta\Ric + d\scal = 0$,
\item $\mu^2(\scal^2 - |\Ric|^2) = \frac{1}{4}\scal$,
\end{enumerate}
where $\scal$ is the scalar curvature of $(M,g)$.
\end{prop}

On the other hand, contracting the second Bianchi identity, we have $2\delta\Ric + d\scal = 0$. Combining this with the second equation in the above proposition, we have the following result.

\begin{coro}\label{coro:scal_is_constant}
If a Riemannian spin manifold $(M, g)$ admits a Ricci Killing spinor, then the scalar curvature $\scal$ is constant. Moreover, $\delta \Ric = 0$ and $|\Ric|^2$ is constant.
\end{coro}

By the definition of Ricci Killing spinors, if $(M, g)$ is Einstein, then Ricci Killing spinors are exactly Killing spinors. 
We consider the converse problem: if a Riemannian spin manifold $(M, g)$ admits a Ricci Killing spinor, under what conditions is it Einstein?

\begin{theoremalpha}\label{theo:inequality_for_Killing_number}
Let $(M, g)$ be an $n$-dimensional Riemannian spin manifold ($n \geq 2$) admitting a Ricci Killing spinor with a constant $\mu$. 
Then the scalar curvature $\scal$ is constant.
If $\scal = 0$ then $(M, g)$ is Ricci-flat and the Ricci Killing spinor is parallel. If $\scal \neq 0$ then the following inequality holds:
\begin{equation}\label{eq:ineq_of_Killing_number}
\mu^2 \geq \frac{n}{4(n-1)\scal}.
\end{equation}
Equality holds if and only if $(M, g)$ is Einstein, namely, the Ricci Killing spinor is a Killing spinor.
\end{theoremalpha}
\begin{proof}
  The constancy of $\scal$ follows from Corollary~\ref{coro:scal_is_constant}.
  When $\scal = 0$, the third equation in Proposition \ref{prop:Basic_properties_of_RKS} implies that $|\Ric|^2 = 0$. Hence, $(M, g)$ is Ricci-flat and the Ricci Killing spinor is parallel.
  When $\scal \neq 0$, by the Cauchy--Schwarz inequality and the third equation in Proposition \ref{prop:Basic_properties_of_RKS}, we have
  \begin{align*}
    \scal^2 \leq n \tr(\Ric^2) = n|\Ric|^2 = n\left(\scal^2 - \frac{\scal}{4\mu^2}\right).
  \end{align*}
  Dividing both sides by $\scal^2 > 0$ and rearranging this inequality, we obtain the inequality \eqref{eq:ineq_of_Killing_number}.
Equality in \eqref{eq:ineq_of_Killing_number} holds if and only if equality holds in the Cauchy--Schwarz inequality, namely, $\Ric = \frac{\scal}{n}\Id$. Since $\scal$ is constant by Corollary \ref{coro:scal_is_constant}, this is equivalent to $(M, g)$ being Einstein.
\end{proof}
\begin{rema}
  As in the case of Killing spinors, the constant $\mu$ in the definition of Ricci Killing spinors can be normalized to $\pm \frac{1}{2}$ by rescaling the metric.
  Under the rescaling and assumption that $\scal > 0$, the inequality \eqref{eq:ineq_of_Killing_number} can be rewritten as $\scal \geq  n/(n-1)$.
\end{rema}

As seen in Section \ref{sec:GKS}, a Riemannian spin manifold admitting a generalized Killing spinor may be locally reducible. 
Proposition \ref{prop:GKS_on_product_manifolds} is also true for Ricci Killing spinors.
Indeed, since a Riemannian spin manifold admitting a parallel spinor is Ricci-flat, the Ricci tensor of the product manifold $M \times N$ in Proposition \ref{prop:GKS_on_product_manifolds} is given by $\Ric(X + Y) = \Ric_M(X)$ for $X \in TM$ and $Y \in TN$.
In particular, if $M$ admits a real Killing spinor, then $M \times N$ admits a Ricci Killing spinor which is not a Killing spinor.
\begin{prop}\label{prop:RKS_on_product_manifolds}
  Let $M$ and $N$ be Riemannian spin manifolds, with $M$ of even dimension.
  If $M$ admits a Ricci Killing spinor and $N$ admits a parallel spinor, then the product manifold $M \times N$ admits a Ricci Killing spinor.
\end{prop}
On the other hand, if a manifold admitting a Ricci Killing spinor is locally reducible, then one of the factors must be Ricci-flat.
\begin{theo}\label{theo:irreducibility_for_RKS}
Let $(M, g)$ be a Riemannian spin manifold. Suppose that there exists an open subset $U \subset M$ isometric to a Riemannian product $U_1 \times U_2$ of two non-Ricci-flat Riemannian manifolds. Then $(M, g)$ does not admit a Ricci Killing spinor.
\end{theo}
\begin{proof}
  If $(M, g)$ admits a Ricci Killing spinor $\psi$, then by Proposition \ref{prop:curv_to_GKS}, we have
  \[\frac{1}{2}R(X \wedge Y) \cdot \psi = \mu((\nabla_X \Ric)(Y) - (\nabla_Y \Ric)(X)) \cdot \psi - 2\mu^2(\Ric(X) \wedge \Ric(Y))\cdot\psi.\]
  For all $X \in \Gamma(TU_1)$ and $Y \in \Gamma(TU_2)$, the left-hand side and the first term on the right-hand side vanish because $R(X,Y) = 0$ and $(\nabla_X \Ric)(Y) = 0 = (\nabla_Y \Ric)(X)$.
  Combining this with $g(\Ric(X), \Ric(Y)) = 0$, we have
  \begin{equation}\label{eq:irreducibility_RKS}0 = (\Ric(X) \wedge \Ric(Y))\cdot\psi = \Ric(X) \cdot \Ric(Y)\cdot\psi.
  \end{equation}
  Since $U_1$ and $U_2$ are non-Ricci-flat, there exists a point $(p,q) \in U_1 \times U_2$ and vectors $X_p \in T_pU_1, Y_q \in T_qU_2$ such that $\Ric(X_p) \neq 0$ and $\Ric(Y_q) \neq 0$.
  Hence, by \eqref{eq:irreducibility_RKS}, we have $\psi = 0$ at $(p,q)$, which contradicts the fact that $\psi$ has no zeroes (Lemma \ref{lemm:GKS_has_no_zeroes}).
\end{proof}

\section{Ricci Killing spinors on Sasakian \texorpdfstring{$\eta$}{eta}-Einstein manifolds}\label{sec:RKS_on_Sasakian}
\subsection{Sasakian quasi-Killing spinors on Sasakian \texorpdfstring{$\eta$}{eta}-Einstein manifolds}\label{sec:SqKS_on_eta_Einstein}
We first recall some basic facts about Sasakian $\eta$-Einstein manifolds and Sasakian quasi-Killing spinors.
We derive restrictions on the possible types of Sasakian quasi-Killing spinors, which will be used in Section \ref{sec:Rigidity_of_eta_Einstein_admitting_RKS} to study Ricci Killing spinors on Sasakian $\eta$-Einstein manifolds.
\begin{defi}
  Let $(M^{2m+1},g)\,\,(m \geq 1)$ be a Riemannian manifold.
  A \textit{Sasakian structure} on $M$ is a triple $(\xi, \eta, \varphi)$ consisting of a vector field $\xi$, a 1-form $\eta$, and a (1,1)-tensor field $\varphi$ such that the following conditions hold:
  \begin{enumerate}
    \item $\eta(\xi) = 1$,
    \item $\varphi^2 = -\Id + \eta \otimes \xi$,
    \item $g(\varphi(X), \varphi(Y)) = g(X, Y) - \eta(X)\eta(Y)$  \quad for all $X, Y \in TM$,
    \item $(\nabla_X \varphi)(Y) = g(X, Y)\xi - \eta(Y)X$ \quad for all $X, Y \in TM$.
  \end{enumerate}
  A Riemannian manifold $(M, g)$ equipped with a Sasakian structure is called a \textit{Sasakian manifold}.
  In particular, $(M, g, \xi, \eta, \varphi)$ is called a \textit{Sasakian $\eta$-Einstein manifold} if there exist constants $\lambda, \nu \in \mathbb{R}$ such that the Ricci tensor satisfies $\Ric = \lambda g + \nu \eta \otimes \eta$.
\end{defi}

Note that if $M$ is a Sasakian $\eta$-Einstein manifold, the constants $\lambda$ and $\nu$ satisfy the relation $\lambda + \nu = 2m$, and the scalar curvature is given by $\scal = 2m(\lambda + 1)$.

\begin{defi}\label{defi:Sasaki_Eintein_and_trans_Ricciflat}
  Let $(M, g, \xi, \eta, \varphi)$ be a Sasakian $\eta$-Einstein manifold.
  $M$ is called \emph{Sasakian-Einstein} if $\Ric = 2m g$, and \emph{transversely Ricci-flat} if $\Ric = -2g + (2m+2)\eta \otimes \eta$.
\end{defi}

\begin{rema}\label{rema:trans_Ricciflat_and_null}
  Recall that a Sasakian structure is called null if the basic first Chern class vanishes, i.e., $c_1^B(\mathcal F_\xi)=0$ (see \cite{BG08}*{Definition~7.5.24}).
  Transverse Ricci-flatness is closely related to the notion of nullity of Sasakian structures.
  To see this, let $\mathcal{D}=\ker\eta$, and denote by $g^T=g|_{\mathcal{D}}$ the transverse metric and by $\Ric^T$ and $\rho^T$ the corresponding transverse Ricci tensor and Ricci form, respectively.
  If $\Ric=\lambda g+\nu\eta\otimes\eta$, then $\Ric^T=(\lambda+2)g^T$ by \cite{BG08}*{Theorem~7.3.12}, and hence $\rho^T = \frac{\lambda + 2}{2}d\eta$.
  Therefore, since $2\pi c_1^B(\mathcal{F}_\xi)=[\rho^T]_B= \frac{\lambda + 2}{2}[d\eta]_B$ by \cite{BG08}*{Lemma~7.5.18}, if $M$ is transversely Ricci-flat, then $c_1^B(\mathcal{F}_\xi)=0$, i.e., the Sasakian structure is null.
  For compact Sasakian $\eta$-Einstein manifolds, the converse is also true since $[d\eta]_B\neq0$ by \cite{BG08}*{Proposition~7.2.3(iii)}.
\end{rema}

The spinor bundle $\Sigma M$ of a Sasakian spin manifold $(M, g, \xi, \eta, \varphi)$ can be decomposed into a direct sum of subbundles $\Sigma_r$ ($r = 0, 1, \cdots, m$) by the action of the fundamental 2-form $\Phi(X,Y) = g(X, \varphi(Y))$ on $\Sigma M$.
We orient $M$ so that any local orthonormal frame $(e_1,\varphi(e_1),\ldots,e_m,\varphi(e_m),\xi)$, where $e_1,\ldots,e_m\in\ker\eta$, is positively oriented.

\begin{lemm}[{\cite{KF00}*{Lemma~6.2}, reformulated in our convention}]
  \label{lemm:spinor_bundle_decomposition_on_Sasakian_manifold}
  Let $(M^{2m+1},g,\xi, \eta, \varphi)$ be a Sasakian spin manifold, and let $\Phi\in\Omega^2(M)$ be its fundamental $2$-form.
  Then the spinor bundle $\Sigma M$ decomposes into a direct sum of subbundles $\Sigma_r$ ($r = 0, 1, \cdots, m$):
  \[\Phi \cdot |_{\Sigma_r} = i(m-2r)\Id_{\Sigma_r}, \quad \mathrm{rank}_{\C}\Sigma_r = \binom{m}{r} \quad (r = 0, 1, \cdots, m)\]
  \[\xi \cdot |_{\Sigma_0 \oplus \Sigma_2 \oplus \cdots} = -i \Id, \quad \xi \cdot |_{\Sigma_1 \oplus \Sigma_3 \oplus \cdots} = i \Id.\]
  Furthermore, the subbundles $\Sigma_0$ and $\Sigma_m$ can be expressed as
  \begin{align*}
    \Sigma_0&=\left\{\psi\in\Sigma M \mid \varphi(X)\cdot\psi-iX\cdot\psi+\eta(X)\psi=0\quad(\forall X\in TM)\right\},\\
    \Sigma_m&=\left\{\psi\in\Sigma M \mid \varphi(X)\cdot\psi+iX\cdot\psi+(-1)^{m+1}\eta(X)\psi=0\quad(\forall X\in TM)\right\}.
  \end{align*}
\end{lemm}

\begin{rema}\label{rema:convention_KF}
The Clifford multiplication used in \cite{KF00} differs from ours in odd dimensions: the complex volume element acts as $-\Id$ in \cite{KF00}, whereas it acts as $+\Id$ by our convention.
Accordingly, the results from \cite{KF00} are reformulated using our Clifford multiplication.
We also adopt the reverse indexing of the subbundles $\Sigma_r$ compared with that in \cite{KF00}, which leads to simpler formulas in our convention.
\end{rema}

Kim and Friedrich introduced the notion of Sasakian quasi-Killing spinors, which form a special class of generalized Killing spinors on Sasakian manifolds.
We briefly review Sasakian quasi-Killing spinors and their properties, following \cite{KF00}.

\begin{defi}
  A non-zero spinor field $\psi \in \Gamma(\Sigma M)$ on a Sasakian spin manifold $(M, g, \xi, \eta, \varphi)$ is called a \textit{Sasakian quasi-Killing spinor of type $(a,b)$} if it satisfies
  \[ \nabla_X \psi = a X \cdot \psi + b \eta(X)\xi \cdot \psi \quad \text{for all } X \in TM, \]
  where $a, b \in \mathbb{R}$ are constants.
\end{defi}

The following identities, due to Kim and Friedrich, relate a Sasakian quasi-Killing spinor to the Ricci curvature and the fundamental $2$-form. They will be used repeatedly in the sequel.
These identities retain the same form under the change of Clifford multiplication described in Remark~\ref{rema:convention_KF}.

\begin{prop}[{\cite{KF00}*{Lemma~6.4}}]\label{prop:fundamental_identities_for_SqK_spinors}
  Let $\psi \in \Gamma(\Sigma M)$ be a Sasakian quasi-Killing spinor of type $(a,b)$. 
  Then the following identities hold:
  \begin{enumerate}
    \item $\Ric(X) \cdot \psi = (8ma^2 + 4ab)X \cdot \psi + 2b\varphi(X) \cdot \xi \cdot \psi + (2m-8ma^2-4ab)\eta(X)\xi\cdot\psi$
    \item $2b\Phi\cdot\psi = m(1-4a^2-4ab)\xi \cdot \psi$
  \end{enumerate}
  In particular, the scalar curvature $\scal$ and the norm of the Ricci tensor $|\Ric|$ are constant and given by
  \begin{align}
    \scal &= 8m(2m+1)a^2 + 16mab \label{eq:sqk_scal}\\
    |\Ric|^2 &= (8ma^2+4ab)(16m^2a^2+16ma^2+24mab-4m)+8mb^2+4m^2 \notag
  \end{align}
\end{prop}

By definition, Ricci Killing spinors on Sasakian $\eta$-Einstein manifolds ($\Ric = \lambda g + \nu \eta \otimes \eta$) are Sasakian quasi-Killing spinors of type $(\mu\lambda, \mu\nu)$, where $\mu$ is a constant associated with the Ricci Killing spinor.
So, we investigate the properties of Sasakian quasi-Killing spinors on Sasakian $\eta$-Einstein manifolds.

\begin{coro}\label{coro:sqk_on_etaEinstein}
  Let $(M^{2m+1}, g, \xi, \eta, \varphi)$ be a Sasakian $\eta$-Einstein spin manifold with $\Ric = \lambda g + \nu \eta \otimes \eta$. If $M$ admits a Sasakian quasi-Killing spinor $\psi$ of type $(a,b)$, then the constants $a,b$ satisfy the following conditions:
  \begin{enumerate}
    \item when $m = 1$, $(a, b) = \left(\frac{1}{2}, \frac{\lambda - 2}{4}\right)$ or $(a, b) = \left(\frac{1 \pm \sqrt{\lambda + 2}}{2}, \frac{-2 \mp \sqrt{\lambda + 2}}{2}\right)$,
    \item when $m \geq 2$, $a = \pm \frac{1}{2}$.
  \end{enumerate}
\end{coro}
\begin{proof}
  The first statement is a direct consequence of \cite{KF00}*{Theorem~8.3}.
  Note that a Sasakian quasi-Killing spinor of type $(a,b)$ in our convention (see Remark \ref{rema:convention_KF}) corresponds to a Sasakian quasi-Killing spinor of type $(-a,-b)$ in \cite{KF00}.
  We prove the second statement.
  If $b = 0$, then Sasakian quasi-Killing spinors are Killing spinors with Killing number $a$.
  In this case, it is well known that $a = \pm \frac{1}{2}$.
  
  We now prove that if $m \geq 2$ and $b \neq 0$, then $a = \pm \frac{1}{2}$.
  Let $X$ be a vector field in $\ker \eta$.
  By the first identity in Proposition \ref{prop:fundamental_identities_for_SqK_spinors}, we have
  \[(\lambda - 8ma^2 - 4ab)X \cdot \psi = 2b\varphi(X) \cdot \xi \cdot \psi.\]
  If $X$ is a unit vector field in $\ker \eta$, then $\varphi(X)$ is also a unit vector field.
  Hence, taking the norms of both sides gives
  \begin{equation}
    \lambda - 8ma^2 - 4ab = \pm 2b. \label{eq:sqk_on_etaEinstein1}
  \end{equation}
  On the other hand, substituting $\scal = 2m(1+\lambda)$ into \eqref{eq:sqk_scal}, we have $\lambda = 4(2m+1)a^2 + 8ab - 1$.
  By substituting this into \eqref{eq:sqk_on_etaEinstein1}, we have $4a^2 + 4ab - 1 = \pm 2b$.
  Hence, by the second identity in Proposition \ref{prop:fundamental_identities_for_SqK_spinors}, we have $\Phi \cdot \psi = \pm m \xi \cdot \psi$.
  Next, we take the covariant derivative of both sides of this equation along a unit vector field $X \in \ker \eta$.
  Using the formula $X \cdot \Phi \cdot \psi - \Phi \cdot X \cdot \psi = 2\varphi(X) \cdot \psi$ (see \cite{KF00}*{Lemma~6.3}),
  \begin{align*}
    \nabla_X (\Phi \cdot \psi) &= (\nabla_X \Phi) \cdot \psi + \Phi \cdot \nabla_X \psi\\
    &= -X \cdot \xi \cdot \psi + a \Phi \cdot X \cdot \psi\\
    &= -X \cdot \xi \cdot \psi + a (X \cdot \Phi \cdot \psi - 2\varphi(X) \cdot \psi)\\
    &= -X \cdot \xi \cdot \psi + a (X \cdot (\pm m \xi \cdot \psi) - 2\varphi(X) \cdot \psi)\\
    &= (-1 \pm ma) X \cdot \xi \cdot \psi - 2a\varphi(X) \cdot \psi.
  \end{align*}
  On the other hand,
  \begin{align*}
    \nabla_X(\xi \cdot \psi) &= (\nabla_X \xi) \cdot \psi + \xi \cdot \nabla_X \psi\\
    &= -\varphi(X) \cdot \psi + \xi \cdot (a X \cdot \psi + b \eta(X) \xi \cdot \psi)\\
    &= -\varphi(X) \cdot \psi + a \xi \cdot X \cdot \psi\\
    &= -\varphi(X) \cdot \psi - a X \cdot \xi \cdot \psi.
  \end{align*}
  Therefore, differentiating $\Phi \cdot \psi = \pm m \xi \cdot \psi$ along $X$ gives
  \[(-1 \pm 2ma)X \cdot \xi \cdot \psi = -(-2a \pm m)\varphi(X) \cdot \psi.\]
  Taking the norms of both sides, we obtain $|-1 \pm 2ma| = |-2a \pm m|$.
  Squaring both sides gives $a^2 = \frac{1}{4}$.
\end{proof}

\subsection{Ricci Killing spinors on Sasakian \texorpdfstring{$\eta$}{eta}-Einstein manifolds}\label{sec:Rigidity_of_eta_Einstein_admitting_RKS}
We now apply the results in the previous section to Ricci Killing spinors on Sasakian $\eta$-Einstein manifolds.
The following proposition shows that the existence of a Ricci Killing spinor imposes a strong restriction on the geometry of the manifold.
\begin{prop}\label{prop:RKS_on_eta_Einstein}
  Let $(M^{2m+1}, g, \xi, \eta, \varphi)\,\,(m \geq 1)$ be a Sasakian $\eta$-Einstein spin manifold with $\Ric = \lambda g + \nu \eta \otimes \eta$.
  If $M$ admits a Ricci Killing spinor, then $M$ is either transversely Ricci-flat or Sasakian-Einstein.
\end{prop}
\begin{proof}
  We assume that a Sasakian $\eta$-Einstein manifold $M$ admits a Ricci Killing spinor with a constant $\mu$ and $M$ is not a Sasakian-Einstein manifold.
  The Ricci Killing spinor is a Sasakian quasi-Killing spinor of type $(a, b) = (\mu\lambda, \mu\nu)$ and $b \neq 0$.
  By equation \eqref{eq:sqk_scal}, we have
  \begin{equation}
    2m(1+\lambda) = 8m(2m+1)a^2 + 16mab \label{eq:scal_on_etaEinstein_with_SqK}
  \end{equation}
  \begin{enumerate}
    \item If $m \geq 2$, then $a = \pm \frac{1}{2}$ by Corollary \ref{coro:sqk_on_etaEinstein}, and by substituting this into \eqref{eq:scal_on_etaEinstein_with_SqK}, we have
     \begin{equation}
      \lambda^2 = 2m\lambda \pm 4\lambda b. \label{eq:RKS_on_eta_Einstein1}
     \end{equation}
     On the other hand, since $\lambda + \nu = 2m$ in a Sasakian $\eta$-Einstein manifold, we have 
    \[\lambda b = \lambda\mu\nu = a(2m-\lambda) = \pm \frac{1}{2}(2m-\lambda).\]
    By substituting this into \eqref{eq:RKS_on_eta_Einstein1}, we have $(\lambda - 2m)(\lambda + 2) = 0$. Since we assumed that $M$ is not a Sasakian-Einstein manifold, i.e., $\lambda \neq 2m$, we have $\lambda = -2$, which implies that $M$ is transversely Ricci-flat.

    \item If $m = 1$, then $a = \frac{1}{2}$ or $a + b = -\frac{1}{2}$ by Corollary \ref{coro:sqk_on_etaEinstein}. 
    For $a = \frac{1}{2}$, we have $\lambda = -2$ by an argument similar to that in the case of $m \geq 2$. 
    For $a + b = -\frac{1}{2}$, since $a = \mu\lambda$ and $b = \mu\nu$, we have $\mu(\lambda + \nu) = -\frac{1}{2}$. Since $\lambda + \nu = 2m = 2$, we have $\mu = -\frac{1}{4}$ and thus
    \[a = -\frac{\lambda}{4}, \quad\quad b = -\frac{\nu}{4} = -\frac{2 - \lambda}{4}.\]
    Substituting these into \eqref{eq:scal_on_etaEinstein_with_SqK}, we have $\lambda^2 = 4$.
    Since $M$ is not a Sasakian-Einstein manifold, we have $\lambda \neq 2$, and thus $\lambda = -2$.
    This means that $M$ is transversely Ricci-flat.
  \end{enumerate}
\end{proof}

We next prove that if $M$ is simply connected, the converse of Proposition \ref{prop:RKS_on_eta_Einstein} is also true.
It is well known that a simply connected Sasakian-Einstein spin manifold admits a Killing spinor.
So, we investigate the existence of Ricci Killing spinors on simply connected transversely Ricci-flat Sasakian manifolds.

Let $(M^{2m+1}, g, \xi, \eta, \varphi)$ be a simply connected transversely Ricci-flat Sasakian spin manifold.
Since $\scal = -2m$ and $|\Ric|^2 = 4m^2 + 8m$, by the third equation of Proposition \ref{prop:Basic_properties_of_RKS}, if $M$ admits a Ricci Killing spinor with a constant $\mu$, then $\mu^2 = \frac{1}{16}$.
Moreover, in this case, the Ricci Killing spinor with $\mu = \pm \frac{1}{4}$ is a Sasakian quasi-Killing spinor of type $(a, b) = (\mp\frac{1}{2}, \pm\frac{1}{2}(1+m))$ and vice versa.

In the three-dimensional case, by Corollary \ref{coro:sqk_on_etaEinstein}, the type $(-\frac{1}{2}, 1)$ does not occur.
Hence, Sasakian quasi-Killing spinors of type $(a, b) = (\frac{1}{2}, -1)$ are the only candidates for Ricci Killing spinors on transversely Ricci-flat Sasakian 3-manifolds.
In general, a simply connected Sasakian $\eta$-Einstein spin $3$-manifold admits Sasakian quasi-Killing spinors of type $\left(\frac12,\frac{\scal-6}{8}\right)$, and the space of such spinors has complex dimension $2$ (see \cite{KF00}*{Theorem~8.4} and Remark~\ref{rema:convention_KF}).

\begin{prop}\label{prop:RKS_on_3dim_trans_Ricciflat}
  Let $(M^3, g, \xi, \eta, \varphi)$ be a simply connected transversely Ricci-flat Sasakian spin manifold.
  Then $M$ admits a Ricci Killing spinor with $\mu = -\frac{1}{4}$ and the dimension of the space of such spinors is 2.
\end{prop}

In particular, if a simply connected transversely Ricci-flat Sasakian 3-manifold is complete, then it is isometric to the three-dimensional Heisenberg group $\mathrm{Nil}^3$ equipped with its standard left-invariant Sasakian metric.
Indeed, a three-dimensional Sasakian $\eta$-Einstein manifold with $\Ric = \lambda g + \nu \eta \otimes \eta$ has constant $\varphi$-sectional curvature $\lambda-1$.
In the transversely Ricci-flat case, $\lambda=-2$, and hence the $\varphi$-sectional curvature is $-3$.
The classification of complete simply connected Sasakian space forms therefore implies that $M$ is isometric to $\mathrm{Nil}^3$ with its standard Sasakian structure (see, for example, \cite{IO24}*{Proposition 7.1}). 
Ricci Killing spinors on three-dimensional unimodular Lie groups, including the Heisenberg group, will be discussed in Section~\ref{sec:RKS_on_3dim_Lie_groups}.

We now consider the case $m\geq 2$.
As observed above, a Ricci Killing spinor with $\mu=\pm\frac14$ is a Sasakian quasi-Killing spinor of type $\left(\mp\frac12,\pm\frac{m+1}{2}\right)$.
Suppose first that there exists a Ricci Killing spinor with $\mu=\frac{1}{4}$. Then the corresponding Sasakian quasi-Killing spinor $\psi$ satisfies $\Phi\cdot\psi=m\xi\cdot\psi$ by Proposition \ref{prop:fundamental_identities_for_SqK_spinors}.
Writing $\psi=\sum_{r=0}^m\psi_r$ with $\psi_r\in\Gamma(\Sigma_r)$, we obtain
\[i(m-2r)\psi_r=m(-1)^{r+1}i\psi_r,\]
by Lemma \ref{lemm:spinor_bundle_decomposition_on_Sasakian_manifold}.
Hence $\psi_r$ can be non-zero only for $r=m$, and this is possible only when $m$ is even. Thus, if $m$ is even, a Ricci Killing spinor with $\mu=\frac14$ must lie in $\Sigma_m$, whereas no such spinor can exist when $m$ is odd.
Similarly, if $\mu=-\frac14$, then a Ricci Killing spinor must lie in $\Sigma_0$ when $m$ is even, while for odd $m$ it may have components in both $\Sigma_0$ and $\Sigma_m$.

The existence of spinors in all the cases allowed above follows from
\cite{KF00}*{Theorem~6.3}, after translating the result into our convention (see Remark \ref{rema:convention_KF}).
We therefore obtain the following proposition.

\begin{prop}\label{prop:RKS_on_trans_Ricciflat}
Let $(M^{2m+1},g,\xi,\eta,\varphi)$ $(m\geq 2)$ be a simply connected transversely Ricci-flat Sasakian spin manifold.
\begin{enumerate}
    \item If $m$ is even, the space of Ricci Killing spinors with $\mu=-\frac14$ is one-dimensional and is contained in $\Gamma(\Sigma_0)$, while the space of Ricci Killing spinors with $\mu=\frac14$ is one-dimensional and is contained in $\Gamma(\Sigma_m)$.
    \item If $m$ is odd, the space of Ricci Killing spinors with $\mu=-\frac14$ is two-dimensional and is the direct sum of a one-dimensional subspace of $\Gamma(\Sigma_0)$ and a one-dimensional subspace of $\Gamma(\Sigma_m)$. There is no Ricci Killing spinor with $\mu=\frac14$.
\end{enumerate}
\end{prop}

Combining Proposition~\ref{prop:RKS_on_eta_Einstein} with the above existence results for transversely Ricci-flat Sasakian manifolds, and using the existence of real Killing spinors on simply connected Sasakian-Einstein spin manifolds, we obtain the first assertion of the following theorem.

\begin{theoremalpha}\label{theo:RKS_on_Sasakian_eta_Einstein}
  Let $(M^{2m+1}, g, \xi, \eta, \varphi)\,\,(m \geq 1)$ be a simply connected Sasakian $\eta$-Einstein spin manifold.
  Then $M$ admits a Ricci Killing spinor if and only if $M$ is either transversely Ricci-flat or Sasakian-Einstein.
  Furthermore, Ricci Killing spinors on transversely Ricci-flat Sasakian spin manifolds are transversely parallel.
\end{theoremalpha}

Many examples of simply connected transversely Ricci-flat Sasakian spin manifolds of dimension at least five are obtained from compact null Sasakian $\eta$-Einstein links of isolated hypersurface singularities defined by weighted homogeneous polynomials (see \cite{BG08} and Remark~\ref{rema:trans_Ricciflat_and_null}).

The last assertion of Theorem~\ref{theo:RKS_on_Sasakian_eta_Einstein} follows from the discussion below.
We now review transversely parallel spinors on Sasakian spin manifolds, following \cite{GH08}.
Let $(M^{2m+1},g,\xi,\eta,\varphi)$ be a Sasakian spin manifold and set $\mathcal{D} = \xi^{\perp} = \ker\eta$.
The normal bundle $\mathcal{D}$ carries the transverse Levi-Civita connection $\nabla^T$, defined by
\[\nabla^T_X Z:=\begin{cases}[\xi,Z]^\mathcal{D},&X=\xi,\\(\nabla_XZ)^\mathcal{D},&X\in \mathcal{D},\end{cases}\qquad Z\in\Gamma(\mathcal{D}),\]
where $(\,\,\cdot\,\,)^\mathcal{D}$ denotes the orthogonal projection onto $\mathcal{D}$.
The spin structure of $M$ induces a spin structure $P_{\Spin}\mathcal{D}$ on $\mathcal{D}$, and the spinor bundle $\Sigma M$ can be identified with $\Sigma \mathcal{D}$.
We equip $\mathcal{D}$ with the induced orientation for which $(e_1,\varphi(e_1),\ldots,e_m,\varphi(e_m))$ is positively oriented.
We take the identification such that the Clifford multiplication on $\Sigma M$ is related to the Clifford multiplication on $\Sigma \mathcal{D}$ by
\[\xi \cdot Z \cdot \psi = Z \cdot_{\mathcal{D}} \psi,\qquad i\xi \cdot \psi = \omega^{\C}_{\mathcal{D}} \cdot \psi,\]
where $Z \in \mathcal{D}$, $\psi \in \Sigma M$, and $\omega^{\C}_{\mathcal{D}}$ is the complex volume element of $\mathcal{D}$.
Under this identification, we regard the spin connection induced by $\nabla^T$ as a connection on $\Sigma M$, which will again be denoted by $\nabla^T$.
The relation between $\nabla^T$ and the spin connection $\nabla$ of $(M,g)$ is given as follows.

\begin{lemm}[{\cite{GH08}*{(2)}}]\label{lemm:transverse_spin_connection_on_Sasakian}
For every $\psi\in\Gamma(\Sigma M)$ and $Z\in \mathcal{D}$, we have
\begin{equation}\label{eq:trans_conn_and_LC_conn}
  \nabla_\xi\psi=\nabla^T_\xi\psi+\frac12\Phi\cdot\psi,\qquad \nabla_Z\psi=\nabla^T_Z\psi-\frac12\xi\cdot\varphi(Z)\cdot\psi.
\end{equation}
\end{lemm}

Equivalently, \eqref{eq:trans_conn_and_LC_conn} can be written as
\[\nabla_X \psi = \nabla^T_X \psi + \frac{1}{2}\eta(X) \Phi \cdot \psi - \frac{1}{2} \xi \cdot \varphi(X) \cdot \psi, \qquad \forall X \in TM.\]

\begin{defi}
A non-zero spinor field $\psi\in\Gamma(\Sigma M)$ is called \emph{transversely parallel} if $\nabla^T_X\psi = 0$ for all $X \in TM$.
\end{defi}

The existence of a transversely parallel spinor forces the transverse Ricci curvature to vanish.

\begin{lemm}\label{lemm:trans_parallel_trans_Ricciflat}
Let $(M^{2m+1},g,\xi,\eta,\varphi)$ be a Sasakian spin manifold. If $M$ admits a non-zero transversely parallel spinor, then $M$ is transversely Ricci-flat.
\end{lemm}

\begin{proof}
A transversely parallel spinor is a $(0,0)$-transversal Killing spinor.
By the integrability formula for transversal Killing spinors on Sasakian manifolds \cite{GH08}*{Theorem~3.4}, for every $Z\in \mathcal{D}$, $\Ric(Z) \cdot \psi = -2Z \cdot \psi$.
Since $\psi$ has no zeroes, we have $\Ric(Z)=-2Z$. 
On the other hand, every Sasakian manifold satisfies $\Ric(\xi)=2m\xi$.
Therefore, we have $\Ric=-2g+(2m+2)\eta\otimes\eta$.
\end{proof}

By direct computation using Lemma \ref{lemm:spinor_bundle_decomposition_on_Sasakian_manifold}, for any spinor fields $\psi_0 \in \Gamma(\Sigma_0)$ and $\psi_m \in \Gamma(\Sigma_m)$, we have
\begin{align*}
  &\frac{1}{2}\eta(X)\Phi \cdot \psi_0 - \frac{1}{2}\xi \cdot \varphi(X) \cdot \psi_0 = -\frac{1}{4}(-2X + 2(m+1)\eta(X)\xi) \cdot \psi_0,\\
  &\frac{1}{2}\eta(X)\Phi \cdot \psi_m - \frac{1}{2}\xi \cdot \varphi(X) \cdot \psi_m = \frac{(-1)^m}{4} (-2X + 2(m+1)\eta(X)\xi) \cdot \psi_m.
\end{align*}
Thus if $\psi_0$ (resp. $\psi_m$) is a transversely parallel spinor, then $\psi_0$ (resp. $\psi_m$) is a Ricci Killing spinor with $\mu=-\frac{1}{4}$ (resp. $\mu=\frac{(-1)^m}{4}$).
Conversely, if $\psi$ is a Ricci Killing spinor on a transversely Ricci-flat Sasakian manifold (cf. Propositions \ref{prop:RKS_on_3dim_trans_Ricciflat}, \ref{prop:RKS_on_trans_Ricciflat}), then $\psi$ is a transversely parallel spinor.

\begin{prop}
Let $(M^{2m+1},g,\xi,\eta,\varphi)\,\,(m \geq 1)$ be a transversely Ricci-flat Sasakian spin manifold.
If $m$ is even, the Ricci Killing spinors with $\mu=-\frac14$ are precisely the transversely parallel spinors in $\Sigma_0$, while those with $\mu=\frac14$ are precisely the transversely parallel spinors in $\Sigma_m$.
If $m$ is odd, the Ricci Killing spinors with $\mu=-\frac14$ are precisely the transversely parallel spinors in $\Sigma_0\oplus\Sigma_m$, and there is no Ricci Killing spinor with $\mu=\frac14$.
\end{prop}

\section{Ricci Killing spinors on 3-dimensional Lie groups}\label{sec:RKS_on_3dim_Lie_groups}
\subsection{Invariant Ricci Killing spinors}\label{sec:invariant_RKS_on_3dim_Lie_groups}
Let $G$ be a connected 3-dimensional Lie group with Lie algebra $\g$, endowed with a left-invariant Riemannian metric $g$.
Fix an orientation of $G$ and a positively oriented orthonormal basis $(e_1,e_2,e_3)$ of $\g$.
By left translations, this basis determines a global positively oriented orthonormal frame on $G$, and hence the positively oriented orthonormal frame bundle is trivialized as $P_{\SO}G \cong G\times \SO(3)$.
Throughout this section, we equip $G$ with the spin structure
\[P_{\Spin}G:=G\times\Spin(3) \longrightarrow G\times\SO(3)\cong P_{\SO}G, \qquad (x,s)\longmapsto (x,\pi(s)),\]
where $\pi\colon\Spin(3)\to\SO(3)$ is the standard double covering.
Accordingly, the spinor bundle is trivialized as $\Sigma G\cong G\times\Sigma_3$.
A spinor field which is constant with respect to this trivialization is called an \emph{invariant spinor}.

We first briefly review some results of Artacho \cite{Art25}.
Write the structure constants of $\g$ with respect to $(e_1,e_2,e_3)$ as
\[[e_i,e_j]=\sum_{k=1}^3 c_{ij}^k e_k.\]
Artacho showed that if a non-zero invariant generalized Killing spinor exists, then every invariant spinor is a generalized Killing spinor with the same associated endomorphism $A$ (see \cite{Art25}*{Lemma~4.1}).
Moreover, Artacho showed that there exist non-zero invariant generalized Killing spinors if and only if
\begin{equation}
  c_{13}^1 + c_{23}^2 = 0,\qquad c_{12}^1 - c_{23}^3 = 0,\qquad c_{12}^2 + c_{13}^3 = 0. \label{eq:condition_GKS_on_3dim_Lie_groups}
\end{equation}
We observe that the above condition is precisely the unimodularity condition for connected Lie groups, i.e., $\tr(\ad_X) = 0$ for all $X \in \g$.

\begin{prop}\label{prop:invariant_GKS_unimodular}
Let $G$ be a connected 3-dimensional Lie group endowed with a left-invariant Riemannian metric.
Then $G$ admits a non-zero invariant generalized Killing spinor if and only if $G$ is unimodular.
\end{prop}
\begin{proof}
With respect to the basis $(e_1,e_2,e_3)$, a straightforward computation gives
\[\tr(\ad_X) = X_1(c^2_{12} + c^3_{13}) + X_2(c^1_{21} + c^3_{23}) + X_3(c^1_{31} + c^2_{32}) \qquad (\forall X \in \g),\]
where $X = X_1 e_1 + X_2 e_2 + X_3 e_3$.
Hence the three conditions \eqref{eq:condition_GKS_on_3dim_Lie_groups} are equivalent to the unimodularity of $G$.
\end{proof}

Since every Ricci Killing spinor is a generalized Killing spinor, Proposition~\ref{prop:invariant_GKS_unimodular} implies that a 3-dimensional Lie group with left-invariant metric admitting an invariant Ricci Killing spinor must be unimodular.

We now restrict our attention to unimodular Lie groups.
Milnor \cite{Mil76} showed that there exists a positively oriented left-invariant orthonormal frame
$(e_1,e_2,e_3)$ of $G$ such that
\[[e_2,e_3]=\lambda_1e_1,\qquad [e_3,e_1]=\lambda_2e_2,\qquad [e_1,e_2]=\lambda_3e_3\]
for some $\lambda_1,\lambda_2,\lambda_3\in\R$.
We call such a frame a \emph{Milnor frame} and the constants $\lambda_1,\lambda_2,\lambda_3$ the \emph{Milnor constants}.
The Milnor constants determine the isomorphism type of $\g$ by their signs.
After reordering the Milnor frame and, if necessary, reversing the orientation, the possible sign patterns are listed in Table~\ref{table:three_dimensional_unimodular_Lie_groups} (see \cite{Mil76}*{p.~307}).
Here, $\ISO(2) = \R^2 \rtimes \SO(2)$ denotes the identity component of the isometry group of the Euclidean plane $\R^2$, and $\ISO_0(1,1) = \R^{1,1} \rtimes \SO_0(1,1)$ denotes the identity component of the isometry group of the Minkowski plane $\R^{1,1}$.
Their Lie algebras are denoted by $\iso(2)$ and $\iso(1,1)$, respectively.
\begin{table}[htbp]
    \centering
    \setlength{\tabcolsep}{2.0em}
    \newcommand{\fw}[1]{\makebox[0.8em][c]{$#1$}}
    \caption{3-dimensional unimodular Lie algebras}
    \begin{tabular}{ccc}
        \toprule
        Signs of $\lambda_1,\lambda_2,\lambda_3$
        & Lie algebra $\g$
        & Associated Lie groups \\
        \midrule
        \fw{+}, \fw{+}, \fw{+}
        & $\su(2)$
        & $\SU(2)$, $\SO(3)$ \\

        \fw{+}, \fw{+}, \fw{-}
        & $\sl(2,\R)$
        & $\SL(2,\R)$, $\SO_0(2,1)$ \\

        \fw{+}, \fw{+}, \fw{0}
        & $\iso(2)$
        & $\ISO(2)$ \\

        \fw{+}, \fw{-}, \fw{0}
        & $\iso(1,1)$
        & $\ISO_0(1,1)$ \\

        \fw{+}, \fw{0}, \fw{0}
        & $\heis_3$
        & Heisenberg group \\

        \fw{0}, \fw{0}, \fw{0}
        & $\R^3$
        & $\R^3$ \\
        \bottomrule
    \end{tabular}
    \label{table:three_dimensional_unimodular_Lie_groups}
\end{table}

For convenience, we set
\[a=-\lambda_1+\lambda_2+\lambda_3,\qquad b=\lambda_1-\lambda_2+\lambda_3,\qquad c=\lambda_1+\lambda_2-\lambda_3.\]
Substituting the structure constants $c_{23}^1=\lambda_1, c_{31}^2=\lambda_2, c_{12}^3=\lambda_3$ into the formula for the associated endomorphism $A$ in \cite{Art25}*{(4.3)}, we obtain
\begin{equation}\label{eq:A_Milnor_frame}
    A
    =
    \frac14
    \begin{pmatrix}
        a & 0 & 0\\
        0 & b & 0\\
        0 & 0 & c
    \end{pmatrix}
\end{equation}
with respect to the Milnor frame $(e_1,e_2,e_3)$.
The Ricci tensor is given by\footnote{We note that \cite{Art25}*{(4.4)} contains a typographical error in the formula for the Ricci tensor.}
\begin{equation}\label{eq:Ric_Milnor_frame}
    \Ric = \frac{1}{2}
    \begin{pmatrix}
        \lambda_1^2-(\lambda_2-\lambda_3)^2 & 0 & 0\\
        0 & \lambda_2^2-(\lambda_1-\lambda_3)^2 & 0\\
        0 & 0 & \lambda_3^2-(\lambda_1-\lambda_2)^2
    \end{pmatrix}
    = \frac{1}{2}
    \begin{pmatrix}
        bc & 0 & 0\\
        0 & ca & 0\\
        0 & 0 & ab
    \end{pmatrix}.
\end{equation}
The invariant spinors are Ricci Killing spinors if and only if there exists a constant $\mu\in\R\setminus\{0\}$ such that $A = \mu \Ric$.
By \eqref{eq:A_Milnor_frame} and \eqref{eq:Ric_Milnor_frame}, this is equivalent to
\begin{equation}\label{eq:RKS_Milnor_system}
  a=  2\mu bc,\qquad b = 2\mu ca,\qquad c = 2\mu ab.
\end{equation}
Thus, the classification of invariant Ricci Killing spinors is reduced to solving the algebraic system \eqref{eq:RKS_Milnor_system}.
Multiplying the three equations in \eqref{eq:RKS_Milnor_system} by $a$, $b$, and $c$, respectively, we obtain
\[a^2 = 2\mu abc,\qquad b^2 = 2\mu abc,\qquad c^2 = 2\mu abc.\]
Hence, $a^2 = b^2 =c ^2$.
If $a = b = c = 0$, then
\[
    \lambda_1=\lambda_2=\lambda_3=0,
\]
and hence $\g$ is commutative by Table~\ref{table:three_dimensional_unimodular_Lie_groups}.
In this case, $A = \Ric = 0$, so every invariant spinor is parallel and is a Ricci Killing spinor for every $\mu\in\R\setminus\{0\}$.
Suppose now that $a,b,c$ are non-zero.
Since $a^2=b^2=c^2$, there are two cases up to a permutation of $a,b,c$.
\begin{enumerate}
  \item First, if $a = b = c$, then
  \[\lambda_1 = \lambda_2 = \lambda_3 =: \lambda \neq 0,\]
  and Table~\ref{table:three_dimensional_unimodular_Lie_groups} shows that $\g\cong\su(2)$. In this case,
  \[A = \frac{\lambda}{4}\Id,\qquad \Ric = \frac{\lambda^2}{2}\Id,\]
  and thus
  \[A = \frac{1}{2\lambda}\Ric.\]
  Therefore, every invariant spinor is a Ricci Killing spinor with $\mu = \frac{1}{2\lambda}$.
  Moreover, it is a Killing spinor with Killing number $\frac{\lambda}{4}$.
  
  \item Next, suppose that exactly one of $a,b,c$ has a sign different from the other two. 
  Up to a permutation, we may assume that $a = -b = -c$.
  It follows that
  \[\lambda_2 = \lambda_3 = 0,\qquad \lambda_1 =: \lambda \neq 0.\]
  Hence, by Table~\ref{table:three_dimensional_unimodular_Lie_groups}, $\g$ is isomorphic to the 3-dimensional Heisenberg Lie algebra $\heis_3$.
  In this case,
  \[A = \frac{1}{4}
      \begin{pmatrix}
        -\lambda & 0 & 0\\
        0 & \lambda & 0\\
        0 & 0 & \lambda
      \end{pmatrix},
      \qquad \Ric = \frac{1}{2}
      \begin{pmatrix}
          \lambda^2 & 0 & 0\\
          0 & -\lambda^2 & 0\\
          0 & 0 & -\lambda^2
      \end{pmatrix},\]
  and hence
  \[A = -\frac{1}{2\lambda}\Ric.\]
  Thus, every invariant spinor is a Ricci Killing spinor with $\mu = -\frac{1}{2\lambda}$. 
  Since $A$ is not a multiple of the identity, these spinors are not Killing spinors.
\end{enumerate}
We have proved the following classification.

\begin{theo}\label{theo:invariant_RKS_on_3dim_Lie_groups}
Let $G$ be a connected 3-dimensional Lie group endowed with a left-invariant Riemannian metric. Then $G$ admits a non-zero invariant Ricci Killing spinor if and only if one of the following holds:
\begin{enumerate}
    \item $\g$ is commutative. In this case, every invariant spinor is parallel and thus a Ricci Killing spinor.
    \item $\g\cong\su(2)$ and the Milnor constants satisfy $\lambda_1 = \lambda_2 = \lambda_3 =: \lambda \neq 0$. 
    In this case, every invariant spinor is a Ricci Killing spinor with constant $\mu = \frac{1}{2\lambda}$ and is a Killing spinor with Killing number $\frac{\lambda}{4}$.
    \item $\g \cong \heis_3$.
    In this case, every invariant spinor is a Ricci Killing spinor which is not a Killing spinor.
\end{enumerate}
\end{theo}

\subsection{Non-invariant Ricci Killing spinors on 3-dimensional unimodular Lie groups}\label{sec:non-invariant_RKS_on_3dim_Lie_groups}
We now consider Ricci Killing spinors which are not necessarily invariant.
Throughout this subsection, we assume that $G$ is simply connected and unimodular, and fix a positively oriented Milnor frame $(e_1,e_2,e_3)$.

Using the trivialization $\Sigma G \cong G \times \Sigma_3$ fixed in the previous subsection, we identify a spinor field $\psi \in \Gamma(\Sigma G)$ with the corresponding smooth map $\psi\colon G \to \Sigma_3$, and use the same symbol $\psi$ for both.
With this identification, the spin connection is given by
\[\nabla_X\psi = X(\psi) + \Lambda(X)\cdot\psi,\]
where $\Lambda$ is a left-invariant endomorphism of $TG$.
With respect to the Milnor frame, we have
\[\Lambda = \frac{1}{4}\begin{pmatrix} a & 0 & 0 \\ 0 & b & 0 \\ 0 & 0 & c \end{pmatrix}.\]

\begin{rema}
  If $\psi$ is invariant, then $X(\psi) = 0$ for any vector field $X$, and hence $\nabla_X\psi = \Lambda(X)\cdot\psi$.
  Since $\Lambda$ is symmetric, it is precisely the associated endomorphism of any invariant generalized Killing spinor.
  This explains the expression for $A$ in \eqref{eq:A_Milnor_frame}.
\end{rema}

For a fixed $\mu \in \R \setminus \{0\}$, we define a connection $\nabla^\mu$ on $\Sigma G$ by
\[\nabla^\mu_X\psi := \nabla_X\psi - \mu\Ric(X)\cdot\psi = X(\psi) + \Lambda^\mu(X)\cdot\psi,\]
where $\Lambda^\mu := \Lambda - \mu\Ric$.
Then $\psi$ is a Ricci Killing spinor with constant $\mu$ if and only if it is parallel with respect to $\nabla^\mu$.
By \eqref{eq:Ric_Milnor_frame}, $\Lambda^\mu$ is diagonal with respect to the Milnor frame.
We write
\[\Lambda^\mu = \begin{pmatrix} a' & 0 & 0 \\ 0 & b' & 0 \\ 0 & 0 & c' \end{pmatrix},\]
where
\begin{equation}\label{eq:abc_prime}
a' = \frac{1}{4}(a - 2\mu bc),\qquad b' = \frac{1}{4}(b - 2\mu ca),\qquad c' = \frac{1}{4}(c - 2\mu ab).
\end{equation}
Let $R^\mu$ denote the curvature of $\nabla^\mu$.
Since $\Lambda^\mu$ is left-invariant, for every left-invariant vector field $X$, the components of $\Lambda^\mu(X)$ with respect to the Milnor frame are constant.
Hence, for left-invariant vector fields $X,Y$, we have
\[R^\mu(X,Y)\psi = \Lambda^\mu(X)\cdot\Lambda^\mu(Y)\cdot\psi - \Lambda^\mu(Y)\cdot\Lambda^\mu(X)\cdot\psi - \Lambda^\mu([X,Y])\cdot\psi.\]
Using $e_2\cdot e_3 = e_1$, $e_3\cdot e_1 = e_2$, and $e_1\cdot e_2 = e_3$, we obtain
\[R^\mu(e_2,e_3) = (2b'c' - \lambda_1a')e_1\cdot,\quad R^\mu(e_3,e_1) = (2c'a' - \lambda_2b')e_2\cdot,\quad R^\mu(e_1,e_2) = (2a'b' - \lambda_3c')e_3\cdot.\]

Since Clifford multiplication by a non-zero vector is invertible and $G$ is simply connected, $G$ admits a Ricci Killing spinor with constant $\mu$ if and only if $R^\mu = 0$.
Hence, this is equivalent to
\begin{equation}\label{eq:non_inv_RKS_solveq}
\lambda_1 a' = 2b'c',\qquad \lambda_2 b' = 2c'a',\qquad \lambda_3 c' = 2a'b'.
\end{equation}
A Ricci Killing spinor is invariant if and only if $a' = b' = c' = 0$.
Therefore, in what follows, we assume that at least one of $a',b',c'$ is non-zero.
\begin{enumerate}
  \item Suppose first that exactly two of $a',b',c'$ vanish.
  After a cyclic permutation of the Milnor frame if necessary, we may assume that $a' = b' = 0, c' \neq 0$.
  The third equation in \eqref{eq:non_inv_RKS_solveq} implies that $\lambda_3 = 0$.
  Hence,
  \[a = -\lambda_1 + \lambda_2,\qquad b = \lambda_1 - \lambda_2 = -a,\qquad c = \lambda_1 + \lambda_2.\]
  The equation $a' = b' = 0$ is equivalent to $a(1 + 2\mu c) = 0$.
  \begin{enumerate}
    \item Suppose first that $a = 0$.
    Then $\lambda_1 = \lambda_2 \neq 0$, and hence $\g \cong \iso(2)$ and $G$ is diffeomorphic to $\R^3$.
    In this case, \eqref{eq:Ric_Milnor_frame} gives $\Ric = 0$, so the metric is flat and the Ricci Killing spinors are parallel.
    
    \item Suppose next that $a \neq 0$.
    Then $c \neq 0$, and
    \[\mu = -\frac{1}{2c} = -\frac{1}{2(\lambda_1 + \lambda_2)}.\]
    Moreover,
    \[c' = \frac{\lambda_1\lambda_2}{\lambda_1 + \lambda_2}.\]
    Since $c' \neq 0$, we have $\lambda_1\lambda_2 \neq 0$.
    Together with $a \neq 0$ and $c \neq 0$, this implies $|\lambda_1| \neq |\lambda_2|$.
    Conversely, if $\lambda_1\lambda_2 \neq 0$ and $|\lambda_1| \neq |\lambda_2|$, then $\mu = -\frac{1}{2(\lambda_1 + \lambda_2)}$ gives a solution of \eqref{eq:non_inv_RKS_solveq}.
    If $\lambda_1\lambda_2 > 0$, then $\g \cong \iso(2)$, and if $\lambda_1\lambda_2 < 0$, then $\g \cong \iso(1,1)$.
    In both cases,
    \[\mu = -\frac{1}{2(\lambda_1 + \lambda_2)}, \qquad \Ric = \frac{1}{2}\begin{pmatrix} \lambda_1^2 - \lambda_2^2 & 0 & 0 \\ 0 & \lambda_2^2 - \lambda_1^2 & 0 \\ 0 & 0 & -(\lambda_1 - \lambda_2)^2 \end{pmatrix},\]
    so the metric is not Einstein.
    Thus, these Ricci Killing spinors are not Killing spinors.
  \end{enumerate}
  
  \item The case in which exactly one of $a',b',c'$ vanishes is impossible.
  For instance, if $a' = 0$, then the first equation in \eqref{eq:non_inv_RKS_solveq} gives $b'c' = 0$, which is a contradiction. 
  
  \item It remains to consider the case $a'b'c' \neq 0$.
  Multiplying the three equations in \eqref{eq:non_inv_RKS_solveq}, we obtain $\lambda_1\lambda_2\lambda_3 = 8a'b'c'$.
  In particular, $\lambda_1,\lambda_2,\lambda_3$ are all non-zero.
  Substituting this identity into \eqref{eq:non_inv_RKS_solveq}, we obtain
  \begin{equation}\label{eq:non_inv_RKS_solveq_allnon-zero}
    4(a')^2 = \lambda_2\lambda_3,\qquad 4(b')^2 = \lambda_1\lambda_3,\qquad 4(c')^2 = \lambda_1\lambda_2.
  \end{equation}
  Hence, $\lambda_1,\lambda_2,\lambda_3$ have the same sign, and therefore $\g \cong \su(2)$.
  Using
  \begin{equation}\label{eq:su2_lambda}
    \lambda_1 = \frac{b+c}{2},\qquad \lambda_2 = \frac{c+a}{2},\qquad \lambda_3 = \frac{a+b}{2},
  \end{equation}
  the equations \eqref{eq:non_inv_RKS_solveq_allnon-zero} can be rewritten as
  \[(a - 2\mu bc)^2 = (a+b)(a+c),\quad (b - 2\mu ca)^2 = (b+a)(b+c),\quad (c - 2\mu ab)^2 = (c+a)(c+b).\]
  Expanding these equations and comparing them, we obtain $b^2c^2 = a^2c^2 = a^2b^2$.
  If one of $a,b,c$ is zero, these equalities imply that at least two of them vanish.
  Combining this with \eqref{eq:su2_lambda} implies that one of $\lambda_1,\lambda_2,\lambda_3$ vanishes, which is a contradiction.
  Hence $abc \neq 0$, and therefore $a^2 = b^2 = c^2$.
  \begin{enumerate}
    \item If one of $a,b,c$ has a sign different from the other two, then two of $\lambda_1, \lambda_2, \lambda_3$ vanish, which is a contradiction.
    \item If $a = b = c$, then $\lambda_1 = \lambda_2 = \lambda_3 =: \lambda \neq 0$.
    In this case, we have 
    \[a' = b' = c' = \frac{\lambda}{4}(1-2\mu\lambda),\]
    and \eqref{eq:non_inv_RKS_solveq} gives $\lambda a' = 2(a')^2$.
    Since $a' \neq 0$, we obtain
    \[\mu = -\frac{1}{2\lambda}.\]
    Since $\Ric = \frac{\lambda^2}{2}\Id$, these Ricci Killing spinors are Killing spinors with Killing number $-\lambda/4$.
  \end{enumerate}
\end{enumerate}

We have proved the following classification.

\begin{theo}\label{theo:non_invariant_RKS_on_3dim_unimodular_Lie_groups}
Let $G$ be a simply connected 3-dimensional unimodular Lie group endowed with a left-invariant Riemannian metric.
Then $G$ admits a non-zero non-invariant Ricci Killing spinor if and only if one of the following holds:
\begin{enumerate}
    \item $\g \cong \su(2)$ and the Milnor constants satisfy $\lambda_1 = \lambda_2 = \lambda_3 =: \lambda \neq 0$.
    In this case, $\mu = -\frac{1}{2\lambda}$, and the non-invariant Ricci Killing spinors are Killing spinors with Killing number $-\frac{\lambda}{4}$.

    \item $\g \cong \iso(2)$.
    If the two non-zero Milnor constants are equal, then the metric is flat and the Ricci Killing spinors are parallel.
    Otherwise, the Ricci Killing spinors are not Killing spinors.

    \item $\g \cong \iso(1,1)$ and the sum of the two non-zero Milnor constants is non-zero.
    In this case, the Ricci Killing spinors are not Killing spinors.
\end{enumerate}
In both the non-flat $\iso(2)$ case and the $\iso(1,1)$ case, after a cyclic permutation of the Milnor frame, let $\lambda_1,\lambda_2$ denote the two non-zero Milnor constants.
Then
\[\mu = -\frac{1}{2(\lambda_1 + \lambda_2)},\qquad \Ric = \frac{1}{2}\begin{pmatrix} \lambda_1^2 - \lambda_2^2 & 0 & 0 \\ 0 & \lambda_2^2 - \lambda_1^2 & 0 \\ 0 & 0 & -(\lambda_1 - \lambda_2)^2 \end{pmatrix}.\]
For each associated endomorphism $A = \mu\Ric$ occurring in the above cases, the space of Ricci Killing spinors has complex dimension $2$.
\end{theo}

Combining Theorems~\ref{theo:invariant_RKS_on_3dim_Lie_groups} and~\ref{theo:non_invariant_RKS_on_3dim_unimodular_Lie_groups}, we obtain a complete classification of Ricci Killing spinors on simply connected 3-dimensional unimodular Lie groups with left-invariant metrics.
The classification is summarized in Table~\ref{table:RKS_on_3dim_unimodular_Lie_groups}.

\subsection{Explicit expressions for Ricci Killing spinors}\label{sec:explicit_expressions_RKS_on_3dim_Lie_groups}
In this subsection, we give explicit expressions for the non-invariant non-Killing Ricci Killing spinors obtained in Section~\ref{sec:non-invariant_RKS_on_3dim_Lie_groups}.
These occur on the simply connected Lie groups with Lie algebras $\iso(2)$ and $\iso(1,1)$.

\paragraph{The $\iso(2)$ case.}
We realize the simply connected Lie group $G = \widetilde{\ISO(2)}$ with Lie algebra $\iso(2)$ as $\R^3$ with the group multiplication
\[(x,y,z)\cdot(x',y',z') = (x + x' \cos z - y' \sin z, y + x' \sin z + y' \cos z, z + z').\]
Note that the inverse of $(x,y,z)$ is given by $(-x \cos z - y \sin z, x \sin z -y \cos z, -z)$.
Let $X,Y,Z$ denote the left-invariant vector fields corresponding to the coordinate vectors $\partial_x,\partial_y,\partial_z$ at the identity.
Explicitly,
\[X = \cos z\frac{\partial}{\partial x}+\sin z\frac{\partial}{\partial y},\qquad Y = -\sin z\frac{\partial}{\partial x}+\cos z\frac{\partial}{\partial y},\qquad Z = \frac{\partial}{\partial z}.\]
These vector fields satisfy
\[[X,Y] = 0,\qquad [Y,Z] = X,\qquad [Z,X] = Y.\]
Let $\lambda_1,\lambda_2 > 0$ with $\lambda_1 \neq \lambda_2$, and set
\[e_1 = \frac{1}{\sqrt{2}}(X+Y),\qquad e_2 = -\sqrt{\frac{\lambda_1}{2\lambda_2}}(X-Y),\qquad e_3 = \sqrt{\lambda_1\lambda_2}Z.\]
Then $(e_1,e_2,e_3)$ is a left-invariant frame on $G$ satisfying
\[[e_2,e_3] = \lambda_1 e_1,\qquad [e_3,e_1] = \lambda_2 e_2,\qquad [e_1,e_2] = 0.\]
We equip $G$ with the left-invariant metric and orientation for which $(e_1,e_2,e_3)$ is a positively oriented orthonormal frame.
Then $(e_1,e_2,e_3)$ is a Milnor frame with Milnor constants $(\lambda_1,\lambda_2,0)$.
Explicitly,
\begin{align*}
  e_1 &= \frac{1}{\sqrt{2}}\left((\cos z - \sin z)\frac{\partial}{\partial x} + (\cos z + \sin z)\frac{\partial}{\partial y}\right),\\
  e_2 &= -\sqrt{\frac{\lambda_1}{2\lambda_2}}\left((\cos z + \sin z)\frac{\partial}{\partial x} + (\sin z - \cos z)\frac{\partial}{\partial y}\right),\\
  e_3 &= \sqrt{\lambda_1\lambda_2}\frac{\partial}{\partial z}.  
\end{align*}
For $\mu = -\frac{1}{2(\lambda_1 + \lambda_2)}$, we have
\[\Lambda^\mu = \frac{\lambda_1\lambda_2}{\lambda_1 + \lambda_2}\begin{pmatrix} 0 & 0 & 0 \\ 0 & 0 & 0 \\ 0 & 0 & 1 \end{pmatrix}.\]
Hence, the Ricci Killing spinor equation is equivalent to
\begin{equation}\label{eq:e2_solveq}
  e_1(\psi) = 0,\qquad e_2(\psi) = 0,\qquad e_3(\psi) + \frac{\lambda_1\lambda_2}{\lambda_1 + \lambda_2}e_3\cdot\psi = 0.
\end{equation}
The first two equations imply that $\partial_x\psi = \partial_y\psi = 0$, and hence $\psi$ depends only on $z$.
We identify $\Sigma_3$ with $\C^2$ and choose a representation of the Clifford algebra
\[e_1\cdot = \begin{pmatrix} 0 & 1 \\ -1 & 0 \end{pmatrix},\qquad e_2\cdot = \begin{pmatrix} 0 & i \\ i & 0 \end{pmatrix},\qquad e_3\cdot = \begin{pmatrix} i & 0 \\ 0 & -i \end{pmatrix}.\]
Writing $\psi(z) = \begin{pmatrix} \psi_1(z) & \psi_2(z) \end{pmatrix}^{\top}$, we can rewrite the third equation in \eqref{eq:e2_solveq} as
\[\sqrt{\lambda_1\lambda_2}\frac{d\psi_1}{dz}(z) + \frac{\lambda_1\lambda_2}{\lambda_1 + \lambda_2}i\psi_1(z) = 0,\qquad \sqrt{\lambda_1\lambda_2}\frac{d\psi_2}{dz}(z) - \frac{\lambda_1\lambda_2}{\lambda_1 + \lambda_2}i\psi_2(z) = 0.\]
Therefore, the Ricci Killing spinors are given by
\[\psi(x,y,z) = \begin{pmatrix} C_1\psi_1(z) \\ C_2\psi_2(z) \end{pmatrix},\]
\[\psi_1(z) = \exp\left(-\frac{\sqrt{\lambda_1\lambda_2}}{\lambda_1 + \lambda_2}iz\right), \quad \psi_2(z) = \exp\left(\frac{\sqrt{\lambda_1\lambda_2}}{\lambda_1 + \lambda_2}iz\right).\]
Here, $(C_1,C_2) \in \C^2 \setminus \{(0,0)\}$.

\begin{rema}
The left-invariant metric considered above is defined by requiring the Milnor frame $(e_1,e_2,e_3)$ to be orthonormal.
With respect to the coordinate frame $(\partial_x,\partial_y,\partial_z)$ on $\R^3$, its matrix is
\[
g = \begin{pmatrix} \frac{1}{2}\left(1 + \frac{\lambda_2}{\lambda_1}\right) - \frac{1}{2}\left(1 - \frac{\lambda_2}{\lambda_1}\right)\sin 2z & \frac{1}{2}\left(1 - \frac{\lambda_2}{\lambda_1}\right)\cos 2z & 0 \\ \frac{1}{2}\left(1 - \frac{\lambda_2}{\lambda_1}\right)\cos 2z & \frac{1}{2}\left(1 + \frac{\lambda_2}{\lambda_1}\right) + \frac{1}{2}\left(1 - \frac{\lambda_2}{\lambda_1}\right)\sin 2z & 0 \\ 0 & 0 & \frac{1}{\lambda_1\lambda_2} \end{pmatrix}.
\]
\end{rema}

\paragraph{The $\mathfrak{iso}(1,1)$ case.}
We realize the simply connected Lie group $G = \ISO_0(1,1)$ with Lie algebra $\mathfrak{iso}(1,1)$ as $\R^3$ with the group multiplication
\[(x,y,z)\cdot(x',y',z') = (x + e^{-z}x',y + e^zy',z + z').\]
Note that the inverse of $(x,y,z)$ is given by $(-e^zx,-e^{-z}y,-z)$.
Let $X,Y,Z$ denote the left-invariant vector fields corresponding to the coordinate vectors $\partial_x,\partial_y,\partial_z$ at the identity.
Explicitly,
\[X = e^{-z}\frac{\partial}{\partial x},\qquad Y = e^z\frac{\partial}{\partial y},\qquad Z = \frac{\partial}{\partial z}.\]
These vector fields satisfy
\[[X,Y] = 0,\qquad [Y,Z] = -Y,\qquad [Z,X] = -X.\]
Let $\lambda_1 > 0$ and $\lambda_2 < 0$ with $\lambda_1 + \lambda_2 \neq 0$, and set
\[e_1 = \frac{1}{\sqrt{2}}(X+Y),\qquad e_2 = \sqrt{-\frac{\lambda_1}{2\lambda_2}}(X-Y),\qquad e_3 = \sqrt{-\lambda_1\lambda_2}Z.\]
Then $(e_1,e_2,e_3)$ is a left-invariant frame on $G$ satisfying
\[[e_2,e_3] = \lambda_1e_1,\qquad [e_3,e_1] = \lambda_2e_2,\qquad [e_1,e_2] = 0.\]
We equip $G$ with the left-invariant metric and orientation for which $(e_1,e_2,e_3)$ is a positively oriented orthonormal frame.
Then $(e_1,e_2,e_3)$ is a Milnor frame with Milnor constants $(\lambda_1,\lambda_2,0)$.
Explicitly,
\[e_1 = \frac{1}{\sqrt{2}}\left(e^{-z}\frac{\partial}{\partial x}+e^z\frac{\partial}{\partial y}\right),\qquad e_2 = \sqrt{-\frac{\lambda_1}{2\lambda_2}}\left(e^{-z}\frac{\partial}{\partial x}-e^z\frac{\partial}{\partial y}\right),\qquad e_3 = \sqrt{-\lambda_1\lambda_2}\frac{\partial}{\partial z}.\]
For $\mu = -\frac{1}{2(\lambda_1 + \lambda_2)}$, we have
\[\Lambda^\mu = \frac{\lambda_1\lambda_2}{\lambda_1 + \lambda_2}\begin{pmatrix} 0 & 0 & 0 \\ 0 & 0 & 0 \\ 0 & 0 & 1 \end{pmatrix}\]
with respect to the Milnor frame.
Hence, the Ricci Killing spinor equation is equivalent to
\begin{equation}\label{eq:e11_solveq}
  e_1(\psi) = 0,\qquad e_2(\psi) = 0,\qquad e_3(\psi) + \frac{\lambda_1\lambda_2}{\lambda_1 + \lambda_2}e_3\cdot\psi = 0.
\end{equation}
The first two equations imply that $\partial_x\psi = \partial_y\psi = 0$, and hence $\psi$ depends only on $z$.
Using the same Clifford representation as in the $\iso(2)$ case and writing $\psi(z) = \begin{pmatrix} \psi_1(z) & \psi_2(z) \end{pmatrix}^{\top}$,
we can rewrite the third equation in \eqref{eq:e11_solveq} as
\[\sqrt{-\lambda_1\lambda_2}\frac{d\psi_1}{dz}(z) + \frac{\lambda_1\lambda_2}{\lambda_1 + \lambda_2}i\psi_1(z) = 0,\qquad \sqrt{-\lambda_1\lambda_2}\,\frac{d\psi_2}{dz}(z) - \frac{\lambda_1\lambda_2}{\lambda_1 + \lambda_2}i\psi_2(z) = 0.\]
Therefore, the Ricci Killing spinors are given by
\[\psi(x,y,z) = \begin{pmatrix} C_1\psi_1(z) \\ C_2\psi_2(z) \end{pmatrix},\]
\[\psi_1(z) = \exp\left(\frac{\sqrt{-\lambda_1\lambda_2}}{\lambda_1 + \lambda_2}iz\right), \quad \psi_2(z) = \exp\left(-\frac{\sqrt{-\lambda_1\lambda_2}}{\lambda_1 + \lambda_2}iz\right).\]
Here, $(C_1,C_2) \in \C^2 \setminus \{(0,0)\}$.

\begin{rema}
The left-invariant metric considered above is defined by requiring the Milnor frame $(e_1,e_2,e_3)$ to be orthonormal.
With respect to the coordinate frame $(\partial_x,\partial_y,\partial_z)$ on $\R^3$, its matrix is
\[
g = \begin{pmatrix} \frac{1}{2}\left(1 - \frac{\lambda_2}{\lambda_1}\right)e^{2z} & \frac{1}{2}\left(1 + \frac{\lambda_2}{\lambda_1}\right) & 0 \\ \frac{1}{2}\left(1 + \frac{\lambda_2}{\lambda_1}\right) & \frac{1}{2}\left(1 - \frac{\lambda_2}{\lambda_1}\right)e^{-2z} & 0 \\ 0 & 0 & -\frac{1}{\lambda_1\lambda_2} \end{pmatrix}.
\]
\end{rema}

\section{Ambient metrics of Ricci Killing spinors}\label{sec:Ambient_metrics_of_RKS}
We briefly recall the ambient construction for generalized Killing spinors.
Let $(M,g)$ be a real-analytic Riemannian spin manifold carrying a generalized Killing spinor $\psi$ with a real-analytic associated symmetric endomorphism $A$.
Then there exists a unique metric $g^{\mathcal{Z}}$ of the form 
\[g^\mathcal{Z} = dt^2 + g_t,\qquad g_0 = g,\]
on a sufficiently small neighborhood $\mathcal{Z}$ of $\{0\} \times M$ inside $\R \times M$ such that $(\mathcal{Z}, g^\mathcal{Z})$, endowed with the spin structure induced from $M$, carries a parallel spinor whose restriction to $\{0\} \times M$ is $\psi$ \cite{AMM13} (see also \cite{BHMMM}*{Theorem 8.43}).
For real Killing spinors, this construction recovers B\"ar's cone construction \cite{Bar93}.
For the left-invariant metrics considered in Section~5, both $g$ and $A = \mu\Ric$ are real analytic, so the above construction applies.
In this section, we apply this construction to the non-Killing Ricci Killing spinors obtained in Section~\ref{sec:RKS_on_3dim_Lie_groups}, namely, the invariant ones in the $\heis_3$ case and the non-invariant ones in the $\mathfrak{iso}(2)$ and $\mathfrak{iso}(1,1)$ cases.

Let $W_t$ denote the Weingarten tensor of the hypersurface $M_t = \{t\}\times M$, defined by
\[W_t(X) = -\nabla^\mathcal{Z}_X\partial_t.\]
The initial condition for the Weingarten tensor is $W_0 = 2A$.
If $\psi$ is a Ricci Killing spinor with constant $\mu$, then $A = \mu\Ric$, and therefore $W_0 = 2\mu\Ric$.
Thus, the constructions below provide explicit $4$-dimensional Ricci-flat spin manifolds $\mathcal{Z}$ carrying parallel spinors and containing $3$-dimensional non-Einstein hypersurfaces $M$ whose Weingarten tensor is proportional to the Ricci tensor of $M$.

\subsection{The evolution equations}
To determine the ambient metric $g^\mathcal{Z}$, we use two equations that it necessarily satisfies.
For any $t$, the hypersurface $M_t$ carries a generalized Killing spinor with associated symmetric endomorphism $A_t = \frac{1}{2}W_t$.
Thus each hypersurface $M_t$ satisfies the integrability condition (Remark~\ref{rema:curv_to_GKS_dim3}) of generalized Killing spinors
\begin{equation}
  \frac{1}{2}R^t(X\wedge Y) = \ast_t(d^{\nabla^t}A_t)(X,Y) - 2A_t(X)\wedge A_t(Y), \quad (\forall X,Y \in TM_t \cong TM)\label{eq:ambient_solv_1}
\end{equation}
where $R^t$ is the curvature of $g_t$, $\ast_t$ is the Hodge star operator of $g_t$, and $d^{\nabla^t}$ is the exterior covariant derivative with respect to the Levi-Civita connection $\nabla^t$ of $g_t$.
Furthermore, the ambient manifold $(\mathcal{Z},g^\mathcal{Z})$ is Ricci-flat, 
\begin{equation}
  \Ric^\mathcal{Z} = 0. \label{eq:ambient_solv_2}
\end{equation}
We will solve the system of equations \eqref{eq:ambient_solv_1} and \eqref{eq:ambient_solv_2} for the three cases: $\heis_3$, $\iso(2)$, and $\iso(1,1)$.

Let $G$ be a connected three-dimensional unimodular Lie group equipped with an orientation and a left-invariant Riemannian metric $g$.
Let $(e_1,e_2,e_3)$ be a Milnor frame, so that
\[[e_2,e_3] = \lambda_1e_1,\qquad [e_3,e_1] = \lambda_2e_2,\qquad [e_1,e_2] = \lambda_3e_3,\]
see Section~\ref{sec:invariant_RKS_on_3dim_Lie_groups}.
Let $(\theta^1,\theta^2,\theta^3)$ denote the dual coframe.
We look for the ambient metric on $\mathcal{Z} = I \times G$, where $I$ is an interval containing $0$, of the form
\[g^\mathcal{Z} = dt^2 + g_t,\qquad g_t = u_1(t)^2(\theta^1)^2 + u_2(t)^2(\theta^2)^2 + u_3(t)^2(\theta^3)^2,\qquad u_i(0) = 1,\]
where the functions $u_i$ are positive on $I$.
For each $t\in I$, we set
\[e_1^t = \frac{1}{u_1(t)}e_1,\qquad e_2^t = \frac{1}{u_2(t)}e_2,\qquad e_3^t = \frac{1}{u_3(t)}e_3.\]
Then $(e_1^t,e_2^t,e_3^t)$ is a Milnor frame for $g_t$, with Milnor constants
\[\lambda_1(t) = \frac{u_1(t)}{u_2(t)u_3(t)}\lambda_1,\qquad \lambda_2(t) = \frac{u_2(t)}{u_3(t)u_1(t)}\lambda_2,\qquad \lambda_3(t) = \frac{u_3(t)}{u_1(t)u_2(t)}\lambda_3.\]
For simplicity, we put
\[a(t) = -\lambda_1(t) + \lambda_2(t) + \lambda_3(t),\qquad b(t) = \lambda_1(t) - \lambda_2(t) + \lambda_3(t),\qquad c(t) = \lambda_1(t) + \lambda_2(t) - \lambda_3(t).\]
Let $\omega^t$ denote the connection form of the Levi-Civita connection of $g_t$.
Then
\begin{equation}\label{eq:conn_form_of_Gt}
  \omega^t(X) = \frac{1}{2}X_1(t)a(t)e_2^t\wedge e_3^t + \frac{1}{2}X_2(t)b(t)e_3^t\wedge e_1^t + \frac{1}{2}X_3(t)c(t)e_1^t\wedge e_2^t,
\end{equation}
where $X = X_1(t)e_1^t + X_2(t)e_2^t + X_3(t)e_3^t$.
Note that $X \in TG_t \cong TG$ is independent of $t$, but the coefficients $X_i(t)$ depend on $t$ since $(e_1^t,e_2^t,e_3^t)$ depends on $t$.
Here, we identify $\bw^2 TG$ with $\so(TG)$ by
\[(Y\wedge Z)(V)=g_t(Y,V)Z-g_t(Z,V)Y \qquad (\forall X, Y, V \in TG).\]
The Ricci tensor with respect to $(e_1^t, e_2^t, e_3^t)$ is
\[\Ric^t = \frac{1}{2}\begin{pmatrix}
  b(t)c(t) & 0 & 0 \\
  0 & c(t)a(t) & 0 \\
  0 & 0 & a(t)b(t)
\end{pmatrix}.\]

Since the restriction of the parallel spinor on $\mathcal{Z}$ to $G_t = \{t\}\times G$ is a generalized Killing spinor with associated symmetric endomorphism $A_t = \frac{1}{2}W_t$, we obtain from \cite{BGM05}*{Proposition~4.1}, for any $X,Y \in TG_t \cong TG$,
\begin{equation}\label{eq:variation_of_gt}
  \frac{\partial}{\partial t}g_t(X,Y) = -2g_t(W_t(X),Y) = -4g_t(A_t(X),Y).
\end{equation}
Hence the endomorphism $A_t$ is diagonal with respect to $(e_1, e_2, e_3)$.
Writing $A_t(e_i) = h_i(t)e_i$, we obtain
\begin{equation}\label{eq:relation_h_and_u}
  h_i(t) = -\frac{u_i'(t)}{2u_i(t)} = -\frac{1}{2}\frac{d}{dt}\log u_i(t).
\end{equation}

In dimension three, the curvature operator is determined by the Ricci tensor, and the above formula for $\Ric^t$ gives
\begin{align*}
  R^t(e_2\wedge e_3) &= \frac{1}{4}\bigl(b(t)c(t) - c(t)a(t) - a(t)b(t)\bigr)e_2\wedge e_3,\\
  R^t(e_3\wedge e_1) &= \frac{1}{4}\bigl(-b(t)c(t) + c(t)a(t) - a(t)b(t)\bigr)e_3\wedge e_1,\\
  R^t(e_1\wedge e_2) &= \frac{1}{4}\bigl(-b(t)c(t) - c(t)a(t) + a(t)b(t)\bigr)e_1\wedge e_2.
\end{align*}

Next, we compute the exterior covariant derivative of $A_t$.
By \eqref{eq:conn_form_of_Gt}, we have
\begin{align*}
  \nabla^t_{e_1}e_1 &= 0, & \nabla^t_{e_1}e_2 &= \frac{u_1(t)u_2(t)}{2 u_3(t)}a(t)e_3, & \nabla^t_{e_1}e_3 &= -\frac{u_1(t)u_3(t)}{2 u_2(t)}a(t)e_2,\\
  \nabla^t_{e_2}e_1 &= -\frac{u_1(t)u_2(t)}{2 u_3(t)}b(t)e_3, & \nabla^t_{e_2}e_2 &= 0, & \nabla^t_{e_2}e_3 &= \frac{u_2(t)u_3(t)}{2 u_1(t)}b(t)e_1,\\
  \nabla^t_{e_3}e_1 &= \frac{u_1(t)u_3(t)}{2 u_2(t)}c(t)e_2, & \nabla^t_{e_3}e_2 &= -\frac{u_2(t)u_3(t)}{2 u_1(t)}c(t)e_1, & \nabla^t_{e_3}e_3 &= 0.
\end{align*}
Hence, we obtain
\begin{align*}
  d^{\nabla^t}A_t(e_2,e_3) &= \frac{u_2(t)u_3(t)}{2u_1(t)}\bigl(b(t)(h_3(t)-h_1(t))+c(t)(h_2(t)-h_1(t))\bigr)e_1,\\
  d^{\nabla^t}A_t(e_3,e_1) &= \frac{u_3(t)u_1(t)}{2u_2(t)} \bigl(c(t)(h_1(t)-h_2(t))+a(t)(h_3(t)-h_2(t))\bigr)e_2,\\
  d^{\nabla^t}A_t(e_1,e_2) &= \frac{u_1(t)u_2(t)}{2u_3(t)} \bigl(a(t)(h_2(t)-h_3(t))+b(t)(h_1(t)-h_3(t))\bigr)e_3.
\end{align*}
Since
\[\ast_te_1 = \frac{u_1(t)}{u_2(t)u_3(t)}e_2\wedge e_3,\qquad \ast_te_2 = \frac{u_2(t)}{u_3(t)u_1(t)}e_3\wedge e_1,\qquad \ast_te_3 = \frac{u_3(t)}{u_1(t)u_2(t)}e_1\wedge e_2,\]
we have
\begin{align*}
  \ast_t(d^{\nabla^t}A_t)(e_2,e_3) &= \frac{1}{2}\bigl(b(t)(h_3(t)-h_1(t))+c(t)(h_2(t)-h_1(t))\bigr)e_2\wedge e_3,\\
  \ast_t(d^{\nabla^t}A_t)(e_3,e_1) &= \frac{1}{2}\bigl(c(t)(h_1(t)-h_2(t))+a(t)(h_3(t)-h_2(t))\bigr)e_3\wedge e_1,\\
  \ast_t(d^{\nabla^t}A_t)(e_1,e_2) &= \frac{1}{2}\bigl(a(t)(h_2(t)-h_3(t))+b(t)(h_1(t)-h_3(t))\bigr)e_1\wedge e_2.
\end{align*}
Moreover,
\begin{align*}
  A_t(e_2)\wedge A_t(e_3) &= h_2(t)h_3(t)e_2\wedge e_3,\\
  A_t(e_3)\wedge A_t(e_1) &= h_3(t)h_1(t)e_3\wedge e_1,\\
  A_t(e_1)\wedge A_t(e_2) &= h_1(t)h_2(t)e_1\wedge e_2.
\end{align*}
Therefore, the integrability condition \eqref{eq:ambient_solv_1} is equivalent to
\begin{align}
  \frac{1}{8}\bigl(b(t)c(t)-c(t)a(t)-a(t)b(t)\bigr)
  &= \frac{1}{2}\bigl(b(t)(h_3(t)-h_1(t))+c(t)(h_2(t)-h_1(t))\bigr)-2h_2(t)h_3(t), \label{eq:ambient_int_1}\\
  \frac{1}{8}\bigl(-b(t)c(t)+c(t)a(t)-a(t)b(t)\bigr)
  &= \frac{1}{2}\bigl(c(t)(h_1(t)-h_2(t))+a(t)(h_3(t)-h_2(t))\bigr)-2h_3(t)h_1(t), \label{eq:ambient_int_2}\\
  \frac{1}{8}\bigl(-b(t)c(t)-c(t)a(t)+a(t)b(t)\bigr)
  &= \frac{1}{2}\bigl(a(t)(h_2(t)-h_3(t))+b(t)(h_1(t)-h_3(t))\bigr)-2h_1(t)h_2(t). \label{eq:ambient_int_3}
\end{align}

Next, we compute the Ricci tensor of $g^\mathcal{Z}$.
It suffices to solve $\Ric^\mathcal{Z}(X,Y)=0$ for all $X,Y \in TG$ (see \cite{AMM13}).
By \cite{BGM05}*{Proposition~4.1}, for any tangent vectors $X,Y\in TG$, we have
\begin{equation}\label{eq:ambient_Ric_tangent}
  \Ric^\mathcal{Z}(X,Y) = \Ric^t(X,Y) + 8g_t(A_t(X),A_t(Y)) - 4\tr(A_t)g_t(A_t(X),Y) + 2\frac{\partial}{\partial t}g_t(A_t(X),Y).
\end{equation}
Since both $\Ric^t$ and $A_t$ are diagonal with respect to $(e_1,e_2,e_3)$, the equations $\Ric^\mathcal{Z}(X,Y)=0$ are equivalent to
\begin{align}
  h_1'(t) &= 2H(t)h_1(t)-\frac{1}{4}b(t)c(t), \label{eq:ambient_evol_1}\\
  h_2'(t) &= 2H(t)h_2(t)-\frac{1}{4}c(t)a(t), \label{eq:ambient_evol_2}\\
  h_3'(t) &= 2H(t)h_3(t)-\frac{1}{4}a(t)b(t), \label{eq:ambient_evol_3}
\end{align}
where $H(t) = \tr(A_t) = h_1(t)+h_2(t)+h_3(t)$.

\subsection{The \texorpdfstring{$\heis_3$}{heis3} case}
In this case, the Milnor constants are $\lambda_1 \neq 0, \lambda_2 = \lambda_3 = 0$.
Thus, we have
\[\lambda_1(t)=\frac{u_1(t)}{u_2(t)u_3(t)}\lambda_1,\qquad \lambda_2(t)=\lambda_3(t)=0,\]
and
\[a(t)=-\lambda_1(t),\qquad b(t)=c(t)=\lambda_1(t).\]
By Section~5.1, the associated endomorphism of an invariant Ricci Killing spinor is
\[A_0 = \frac{1}{4}\begin{pmatrix}
  -\lambda_1 & 0 & 0 \\
  0 & \lambda_1 & 0 \\
  0 & 0 & \lambda_1
\end{pmatrix},\]
and hence
\[h_1(0) = -\frac{\lambda_1}{4},\qquad h_2(0) = h_3(0) = \frac{\lambda_1}{4}.\]
Substituting these relations into \eqref{eq:ambient_int_1}--\eqref{eq:ambient_int_3}, we obtain
\begin{align}
  \frac{3}{8}\lambda_1(t)^2 &= \lambda_1(t)\left(\frac{h_2(t)+h_3(t)}{2}-h_1(t)\right)-2h_2(t)h_3(t), \label{eq:heis_int_1}\\
  -\frac{1}{8}\lambda_1(t)^2 &= \frac{1}{2}\lambda_1(t)\bigl(h_1(t)-h_3(t)\bigr)-2h_3(t)h_1(t), \label{eq:heis_int_2} \\
  -\frac{1}{8}\lambda_1(t)^2 &= \frac{1}{2}\lambda_1(t)\bigl(h_1(t)-h_2(t)\bigr)-2h_1(t)h_2(t). \label{eq:heis_int_3}
\end{align}
Similarly, \eqref{eq:ambient_evol_1}--\eqref{eq:ambient_evol_3} read
\begin{align}
  h_1'(t) &= 2H(t)h_1(t)-\frac{1}{4}\lambda_1(t)^2, \label{eq:heis_evol_1}\\
  h_2'(t) &= 2H(t)h_2(t)+\frac{1}{4}\lambda_1(t)^2, \label{eq:heis_evol_2} \\
  h_3'(t) &= 2H(t)h_3(t)+\frac{1}{4}\lambda_1(t)^2. \label{eq:heis_evol_3}
\end{align}

Since $h_2(t)$ and $h_3(t)$ satisfy the same differential equation with the same initial value, uniqueness implies that $h_2(t) = h_3(t)$, and thus \eqref{eq:heis_int_2} reads
\[\bigl(\lambda_1(t) + 4h_1(t)\bigr)\bigl(\lambda_1(t) - 4h_2(t)\bigr)=0.\]
Hence, for each $t \in I$, we have $\lambda_1(t) = 4h_2(t)$ or $\lambda_1(t) = -4h_1(t)$.
If $\lambda_1(t) = -4h_1(t)$, then \eqref{eq:heis_int_1} gives $(h_1(t) + h_2(t))^2 = 0$, which implies $h_1(t) = -h_2(t)$.
Similarly, if $\lambda_1(t)=4h_2(t)$, then \eqref{eq:heis_int_1} gives $h_2(t)\bigl(h_1(t) + h_2(t)\bigr) = 0$.
Since $\lambda_1(t)\neq0$, we have $h_2(t)\neq0$, and hence $h_1(t)=-h_2(t)$.
Thus, we have $\lambda_1(t) = -4h_1(t) = 4h_2(t)$ for all $t \in I$.

Combining $h_1(t) = -h_2(t) = -h_3(t)$ with the relation \eqref{eq:relation_h_and_u} and the initial condition $u_i(0)=1$, we obtain $u_1(t)u_2(t) = u_1(t)u_3(t) = 1$.
Therefore, since $\lambda_1(t)=\lambda_1 u_1(t)^3$, $\lambda_1(t) = -4h_1(t)$ reads
$2 u_1'(t) = \lambda_1 u_1(t)^4$.
Equivalently,
\[\frac{d}{dt}\bigl(u_1(t)^{-3}\bigr)=-\frac{3\lambda_1}{2}.\]
Using $u_1(0)=1$, we obtain
\[u_1(t)=\left(1-\frac{3\lambda_1}{2}t\right)^{-1/3},\qquad u_2(t)=u_3(t)=\left(1-\frac{3\lambda_1}{2}t\right)^{1/3}.\]
Thus, the ambient metric is
\[g^\mathcal{Z}=dt^2+\left(1-\frac{3\lambda_1}{2}t\right)^{-2/3}(\theta^1)^2+\left(1-\frac{3\lambda_1}{2}t\right)^{2/3}\bigl((\theta^2)^2+(\theta^3)^2\bigr),\]
on an interval $I$ containing $0$ such that $1-\frac{3\lambda_1}{2}t>0$.
A direct computation shows that the above solution also satisfies \eqref{eq:heis_evol_1}--\eqref{eq:heis_evol_3}.

\subsection{The \texorpdfstring{$\iso(2)$}{iso(2)} case}
In this case, the Milnor constants satisfy $\lambda_1,\lambda_2>0$, $\lambda_1\neq\lambda_2$, and $\lambda_3=0$.
Thus, we have
\[\lambda_1(t)=\frac{u_1(t)}{u_2(t)u_3(t)}\lambda_1,\qquad \lambda_2(t)=\frac{u_2(t)}{u_3(t)u_1(t)}\lambda_2,\qquad \lambda_3(t)=0,\]
and
\[a(t)=-\lambda_1(t)+\lambda_2(t),\qquad b(t)=-a(t),\qquad c(t)=\lambda_1(t)+\lambda_2(t).\]
By Section~\ref{sec:non-invariant_RKS_on_3dim_Lie_groups}, the associated endomorphism of a non-invariant Ricci Killing spinor is
\[A_0 = -\frac{1}{4}\begin{pmatrix}
  \lambda_1-\lambda_2 & 0 & 0 \\
  0 & \lambda_2-\lambda_1 & 0 \\
  0 & 0 & -\frac{(\lambda_1-\lambda_2)^2}{\lambda_1+\lambda_2}
\end{pmatrix},\]
and hence
\begin{equation}\label{eq:iso2_initial_h}
  h_1(0)=-\frac{\lambda_1-\lambda_2}{4},\qquad h_2(0)=\frac{\lambda_1-\lambda_2}{4},\qquad h_3(0)=\frac{(\lambda_1-\lambda_2)^2}{4(\lambda_1+\lambda_2)}.
\end{equation}
Substituting these relations into \eqref{eq:ambient_int_1}--\eqref{eq:ambient_int_3}, we obtain
\begin{align}
  \frac{1}{8}\bigl(a(t)^2-2a(t)c(t)\bigr) &= \frac{1}{2}\bigl(-a(t)(h_3(t)-h_1(t))+c(t)(h_2(t)-h_1(t))\bigr)-2h_2(t)h_3(t), \label{eq:iso2_int_1}\\
  \frac{1}{8}\bigl(a(t)^2+2a(t)c(t)\bigr) &= \frac{1}{2}\bigl(c(t)(h_1(t)-h_2(t))+a(t)(h_3(t)-h_2(t))\bigr)-2h_3(t)h_1(t), \label{eq:iso2_int_2}\\
  -\frac{1}{8}a(t)^2 &= \frac{1}{2}a(t)(h_2(t)-h_1(t))-2h_1(t)h_2(t). \label{eq:iso2_int_3}
\end{align}
Similarly, \eqref{eq:ambient_evol_1}--\eqref{eq:ambient_evol_3} read
\begin{align}
  h_1'(t) &= 2H(t)h_1(t)+\frac{1}{4}a(t)c(t), \label{eq:iso2_evol_1}\\
  h_2'(t) &= 2H(t)h_2(t)-\frac{1}{4}a(t)c(t), \label{eq:iso2_evol_2}\\
  h_3'(t) &= 2H(t)h_3(t)+\frac{1}{4}a(t)^2. \label{eq:iso2_evol_3}
\end{align}

Adding \eqref{eq:iso2_evol_1} and \eqref{eq:iso2_evol_2}, we obtain
\[\frac{d}{dt}\bigl(h_1(t)+h_2(t)\bigr)=2H(t)\bigl(h_1(t)+h_2(t)\bigr).\]
Since $h_1(0)+h_2(0)=0$, uniqueness implies that $h_1(t)+h_2(t)=0$.
Then \eqref{eq:iso2_int_3} reads $\bigl(a(t)-4h_1(t)\bigr)^2=0$, and hence $a(t)=4h_1(t)=-4h_2(t)$.
Combining $h_1(t)+h_2(t)=0$ with the relation \eqref{eq:relation_h_and_u} and the initial condition $u_i(0)=1$, we obtain $u_1(t)u_2(t) = 1$.

We set $f(t)=u_1(t)^2$.
Then $u_2(t)^2=f(t)^{-1}$, and hence
\[a(t)=\frac{-\lambda_1f(t)+\lambda_2f(t)^{-1}}{u_3(t)}.\]
Since $a(t)=4h_1(t)$ and $h_1(t)=-\frac{u_1'(t)}{2u_1(t)}$, we obtain
\begin{equation}\label{eq:diff_of_f}
  f'(t)= -a(t)f(t) = \frac{\lambda_1f(t)^2-\lambda_2}{u_3(t)},\qquad f(0)=1.
\end{equation}

We set $q=\sqrt{\lambda_1\lambda_2}$ and introduce a new parameter $s$ by
\[\frac{ds}{dt}=-\frac{q}{u_3(t)}.\]
Since $q>0$ and $u_3(t)>0$, the function $s(t)$ is strictly decreasing and admits an inverse $t=t(s)$.
We write $U_i(s)=u_i(t(s))$ and $F(s)=f(t(s))=U_1(s)^2$.
Then
\[\frac{dF}{ds}=-\frac{\lambda_1F(s)^2-\lambda_2}{q}.\]
To solve this equation explicitly, we need to determine $s(0)$.
For simplicity, we choose $s_0=s(0)>0$ so that
\[\tanh(s_0) = \left\{\begin{aligned}
    &\sqrt{\dfrac{\lambda_1}{\lambda_2}}, & \lambda_1<\lambda_2,\\
    &\sqrt{\dfrac{\lambda_2}{\lambda_1}}, & \lambda_1>\lambda_2.
  \end{aligned}\right.\]
Then
\[F(s) = \left\{\begin{aligned}
  &\sqrt{\dfrac{\lambda_2}{\lambda_1}}\tanh(s), & \lambda_1<\lambda_2,\\
  &\sqrt{\dfrac{\lambda_2}{\lambda_1}}(\tanh(s))^{-1}, & \lambda_1>\lambda_2.
\end{aligned}\right.\]
Hence
\[U_1(s)^2 = \left\{\begin{aligned}
&\dfrac{q}{\lambda_1}\tanh(s), & \lambda_1<\lambda_2,\\
&\dfrac{q}{\lambda_1\tanh(s)}, & \lambda_1>\lambda_2,
\end{aligned}\right.
\qquad
U_2(s)^2=
\left\{\begin{aligned}
&\dfrac{\lambda_1}{q\tanh(s)}, & \lambda_1<\lambda_2,\\
&\dfrac{\lambda_1}{q}\tanh(s), & \lambda_1>\lambda_2.
\end{aligned}\right.\]

It remains to determine $U_3(s)$.
In order to do this, we will find a conserved quantity.
Since $h_1(t)+h_2(t)=0$, we have $H(t)=h_3(t)$, and hence \eqref{eq:iso2_evol_3} reads
\[h_3'(t)=2h_3(t)^2+\frac{1}{4}a(t)^2.\]
Using \eqref{eq:relation_h_and_u}, we obtain
\begin{equation}\label{eq:iso2_u3_second}
  u_3''(t)=-\frac{1}{2}a(t)^2u_3(t).
\end{equation}
We set $\rho(t)=\lambda_1f(t)+\lambda_2f(t)^{-1}$.
By \eqref{eq:diff_of_f}, we have $\rho'(t)=a(t)^2u_3(t)$.
Therefore, $\rho'(t) = -2u_3''(t)$, and hence the quantity $\rho(t) + 2u_3'(t)$ is constant.
From the initial condition
\[u_3'(0) = -2 h_3(0) = -\frac{(\lambda_1-\lambda_2)^2}{2(\lambda_1+\lambda_2)},\]
and $\rho(0) = \lambda_1+\lambda_2$, we see that the constant is
\[\rho(t) + 2u_3'(t) = \rho(0) + 2u_3'(0) = \frac{4q^2}{\lambda_1+\lambda_2}.\]
Thus,
\[u_3'(t)=-\frac{1}{2}\left(\lambda_1f(t)+\lambda_2f(t)^{-1}\right)+\frac{2q^2}{\lambda_1+\lambda_2}.\]
Passing to the parameter $s$, we obtain
\[\frac{d}{ds}\log U_3(s)=-\frac{u_3'(t(s))}{q}=\frac{1}{2q}\left(\lambda_1F(s)+\lambda_2F(s)^{-1}\right)-\frac{2q}{\lambda_1+\lambda_2}.\]
In both cases $\lambda_1<\lambda_2$ and $\lambda_1>\lambda_2$, the formula for $F(s)$ gives
\[\frac{1}{q}\left(\lambda_1F(s)+\lambda_2F(s)^{-1}\right)=\tanh(s)+(\tanh(s))^{-1} = \frac{2}{\tanh(2s)}.\]
Hence
\[\frac{d}{ds}\log U_3(s) = \frac{1}{\tanh(2s)} - \frac{2q}{\lambda_1+\lambda_2}.\]
Since $s_0>0$, by shrinking $I$ if necessary, we may assume that $s(I)\subset(0,\infty)$.
Integrating the above equation and using $U_3(s_0)=u_3(0)=1$, we obtain
\[\log U_3(s)=\frac{1}{2}\log\left(\frac{\sinh(2s)}{\sinh(2s_0)}\right)-\frac{2q}{\lambda_1+\lambda_2}(s-s_0).\]
From the definition of $s_0$, we also have
\[\sinh(2s_0)=\frac{2q}{|\lambda_1-\lambda_2|}.\]
Thus,
\[U_3(s)=\left(\frac{|\lambda_1-\lambda_2|}{2q}\sinh(2s)\right)^{\frac{1}{2}}\exp\left(-\frac{2q}{\lambda_1+\lambda_2}(s-s_0)\right).\]
Finally, the ambient metric can be written as
\begin{align*}
  g^\mathcal{Z} &= \frac{U_3(s)^2}{q^2}ds^2+U_1(s)^2(\theta^1)^2+U_2(s)^2(\theta^2)^2+U_3(s)^2(\theta^3)^2\\
  &= \frac{|\lambda_1 - \lambda_2|}{2q^3}\sinh(2s)\exp\left(-\frac{4q}{\lambda_1 + \lambda_2}(s-s_0)\right)(ds^2 + q^2 (\theta^3)^2)\\
  &\hspace{130pt}+ \left\{
    \begin{aligned}
      &\frac{q}{\lambda_1}\tanh(s) (\theta^1)^2 + \frac{\lambda_1}{q \tanh(s)} (\theta^2)^2 & (\lambda_1 < \lambda_2)\\
      &\frac{q}{\lambda_1 \tanh(s)} (\theta^1)^2 + \frac{\lambda_1}{q} \tanh(s) (\theta^2)^2 & (\lambda_1 > \lambda_2)
    \end{aligned}
    \right.
\end{align*}

The above expression defines a Riemannian metric for all $s\in(0,\infty)$, with the initial hypersurface corresponding to $s=s_0$.
We now determine the corresponding range of $t$.
We have
\[t(s) = -\frac{1}{q}\int_{s_0}^s U_3(s')\,ds' = \frac{1}{q}\int_s^{s_0} U_3(s')\,ds'.\]
Since $U_3$ is bounded on the closed interval $[0,s_0]$, $t(0)$ is finite.
On the other hand, for $s\geq 1$,
\[\sinh(2s)\geq\frac{1-e^{-4}}{2}e^{2s},\]
and hence
\[U_3(s)\geq C e^{\alpha s},\qquad \alpha=1-\frac{2q}{\lambda_1+\lambda_2}>0,\]
for some $C>0$, where the last inequality follows from $\lambda_1+\lambda_2>2\sqrt{\lambda_1\lambda_2}=2q$.
Thus $t(s)\to-\infty$ as $s\to\infty$, and consequently the range of $t$ is
\[I=(-\infty,t(0)),\qquad t(0)=\frac{1}{q}\int_0^{s_0}U_3(s)\,ds > 0.\]

\subsection{The \texorpdfstring{$\iso(1,1)$}{iso(1,1)} case}
In this case, the Milnor constants satisfy $\lambda_1 > 0$, $\lambda_2 < 0$, $\lambda_1 + \lambda_2 \neq 0$, and $\lambda_3 = 0$.
As in the $\iso(2)$ case, we have
\[\lambda_1(t) = \frac{u_1(t)}{u_2(t)u_3(t)}\lambda_1,\qquad \lambda_2(t)=\frac{u_2(t)}{u_3(t)u_1(t)}\lambda_2,\qquad \lambda_3(t)=0,\]
and
\[a(t) = -\lambda_1(t) + \lambda_2(t),\qquad b(t) = -a(t),\qquad c(t) = \lambda_1(t) + \lambda_2(t).\]
The initial values are also given by \eqref{eq:iso2_initial_h}.
Therefore, the equations to be solved are exactly the same as in the $\iso(2)$ case, namely \eqref{eq:iso2_int_1}--\eqref{eq:iso2_int_3} and \eqref{eq:iso2_evol_1}--\eqref{eq:iso2_evol_3}.
By the same argument as in the $\iso(2)$ case, we obtain $a(t) = 4h_1(t) = -4h_2(t)$, and $u_1(t)u_2(t) = 1$.

We set $f(t) = u_1(t)^2$.
Then, as in \eqref{eq:diff_of_f},
\[f'(t) = \frac{\lambda_1f(t)^2-\lambda_2}{u_3(t)},\qquad f(0) = 1.\]
As in the $\iso(2)$ case, we set $q = \sqrt{-\lambda_1\lambda_2}>0$ and introduce a new parameter $s$ by
\[\frac{ds}{dt} = \frac{q}{u_3(t)}.\]
Since $q>0$ and $u_3(t)>0$, the function $s(t)$ is strictly increasing and admits an inverse $t=t(s)$.
We write $U_i(s) = u_i(t(s))$ and $F(s)=f(t(s))=U_1(s)^2$.
Then
\[\frac{dF}{ds} = \frac{\lambda_1F(s)^2 - \lambda_2}{q}.\]
We choose $s_0=s(0)$ so that 
\[\tan(s_0) = \frac{\lambda_1}{q} = \sqrt{\frac{\lambda_1}{-\lambda_2}}, \qquad 0 < s_0 < \frac{\pi}{2}.\]
Then we obtain
\[F(s)=\frac{q}{\lambda_1}\tan(s), \qquad U_1(s)^2=\frac{q}{\lambda_1}\tan(s),\qquad U_2(s)^2=\frac{\lambda_1}{q\tan(s)}.\]

It remains to determine $U_3(s)$.
By the same argument as in the $\iso(2)$ case, we obtain
\[u_3'(t)=-\frac{1}{2}\left(\lambda_1f(t)+\lambda_2f(t)^{-1}\right)-\frac{2q^2}{\lambda_1+\lambda_2}.\]
Passing to the parameter $s$, we obtain
\[\frac{d}{ds}\log U_3(s)=\frac{u_3'(t(s))}{q}=-\frac{1}{2q}\left(\lambda_1F(s)+\lambda_2F(s)^{-1}\right)-\frac{2q}{\lambda_1+\lambda_2}.\]
The formula for $F(s)$ gives
\[\frac{1}{q}\left(\lambda_1F(s)+\lambda_2F(s)^{-1}\right)=\tan(s)-(\tan(s))^{-1}=-\frac{2}{\tan(2s)}.\]
Hence
\[\frac{d}{ds}\log U_3(s)=\frac{1}{\tan(2s)}-\frac{2q}{\lambda_1+\lambda_2}.\]

Since $s_0 \in (0,\pi/2)$, by shrinking $I$ if necessary, we may assume that $s(I) \subset (0,\pi/2)$.
Integrating the above equation and using $U_3(s_0)=u_3(0)=1$, we obtain
\[\log U_3(s)=\frac{1}{2}\log\left(\frac{\sin(2s)}{\sin(2s_0)}\right)-\frac{2q}{\lambda_1+\lambda_2}(s-s_0).\]
From the definition of $s_0$, we have
\[\sin(2s_0)=2q/(\lambda_1-\lambda_2).\]
Thus,
\[U_3(s)=\left(\frac{\lambda_1-\lambda_2}{2q}\sin(2s)\right)^{\frac{1}{2}}\exp\left(-\frac{2q}{\lambda_1+\lambda_2}(s-s_0)\right).\]
Finally, the ambient metric can be written as
\begin{align*}
  g^\mathcal{Z}
  =&\frac{\lambda_1-\lambda_2}{2q^3}\sin(2s)\exp\left(-\frac{4q}{\lambda_1+\lambda_2}(s-s_0)\right)(ds^2+q^2(\theta^3)^2)\\
  &\hspace{170pt}+\frac{q}{\lambda_1}\tan(s)(\theta^1)^2+\frac{\lambda_1}{q\tan(s)}(\theta^2)^2. 
\end{align*}
The above expression defines a Riemannian metric for all $s\in(0,\pi/2)$, with the initial hypersurface corresponding to $s=s_0$.
We now determine the corresponding range of $t$.
We have
\[t(s)=\frac{1}{q}\int_{s_0}^s U_3(s')\,ds'.\]
Since $U_3$ is bounded on the closed interval $[0,\pi/2]$, both $t(0)$ and $t(\pi/2)$ are finite.
Consequently, the range of $t$ is
\[I=\left(t(0),t\left(\frac{\pi}{2}\right)\right),\qquad t(0)=-\frac{1}{q}\int_0^{s_0}U_3(s)\,ds<0,\qquad t\left(\frac{\pi}{2}\right)=\frac{1}{q}\int_{s_0}^{\pi/2}U_3(s)\,ds>0.\]

\begin{rema}
The ambient metrics obtained above in the $\mathfrak{heis}_3$, $\mathfrak{iso}(2)$, and $\mathfrak{iso}(1,1)$ cases can be identified with special cases of the Bianchi II rank-zero, Bianchi $\mathrm{VII}_0$ rank-one, and Bianchi $\mathrm{VI}_0$ rank-one gravitational instantons, respectively, classified in \cite{BEPS10}. Moreover, the Kretschmann scalars computed in \cite{BEPS10}*{(80), (88), and (86)} show that all three ambient metrics are non-flat.
\end{rema}

\section*{Acknowledgements}
The author is deeply grateful to his thesis advisor, Professor Yasushi Homma, for invaluable guidance, constant encouragement, and helpful discussions throughout this work.
The author also thanks Professor Uwe Semmelmann for valuable comments and suggestions.

\bibliography{refs}

\vspace{50pt}

\textsc{Natsuki Imada, Department of Pure and Applied Mathematics, Graduate School of Fundamental Science and Engineering, Waseda University, 3-4-1 Ohkubo, Shinjuku-ku, Tokyo, 169-8555, Japan.} \\
\vspace{-10pt}

\textit{E-mail address}: \texttt{natsuki.imada@akane.waseda.jp}

\end{document}